\documentclass[12pt,reqno]{amsart}

\usepackage{amsmath, amsfonts, amsthm, amssymb, color,  graphicx, mathrsfs, cite}
\usepackage{stmaryrd}

\makeatletter
\@addtoreset{figure}{section}
\makeatother

\theoremstyle{plain}
\newtheorem{Theorem}{Theorem}[section]
\newtheorem{Lemma}{Lemma}[section]
\newtheorem{Proposition}{Proposition}[section]

\theoremstyle{definition}

\theoremstyle{Remark}
\newtheorem{Remark}{Remark}[section]

\numberwithin{equation}{section}
\allowdisplaybreaks

\usepackage{ifpdf}
\ifpdf \usepackage[colorlinks=true, citecolor=blue, linkcolor=blue, urlcolor=blue]{hyperref} \fi

\newcommand{\ga}{\gamma}
\newcommand{\Om}{\Omega}

\newcommand{\eps}{\epsilon}
\newcommand{\va}{\varepsilon}

\newcommand{\IT}{\int_0^t}
\newcommand{\IX}{\int_0^\xi}
\newcommand{\IXX}{\int_{\mathbb{R}_+}}

\begin{document}
\title[The non-viscosity Navier-Stokes equations]
{ Asymptotic stability   for the inflow problem of
the heat-conductive ideal gas  without viscosity}

\author{Meichen Hou$^{1,2}$}
\address{1. School of mathematical sciences, University of Chinese Academy of Sciences, Beijing,
100049, China.\\
   2. Institute of Applied Mathematics, AMSS, Beijing 100190, China}
\email{meichenhou@amss.ac.cn}

\author{Lili Fan$^{3}$}
\address{3. School of Mathematics and Computer Science, Wuhan Polytechnic University, Wuhan 430023, China}
\email{fll810@live.cn}

\date{} \maketitle


\begin{abstract}
This paper is devoted to  studying
 the inflow problem for an ideal polytropic model with non-viscous gas in one-dimensional half space. We showed the existence of the boundary layer in different areas. By employing the energy method, we also proved the
unique global-in-time solution existed and the  asymptotic stability of both the boundary layer and  the superposition with
 the $3-$rarefaction wave under some smallness conditions.

\bigbreak
\end{abstract}

\tableofcontents

\section{Introduction }
\setcounter{equation}{0}

In this paper, we consider  the system of heat-conductive ideal gas without viscosity in one-dimensional:
\begin{equation}\label{1.1}
\left\{
\begin{aligned}
&\rho_t+(\rho u)_x=0,\\
&(\rho u)_t+(\rho u^2+p)_x=0,\\
&(\rho(e+\frac{u^2}{2}))_t+(\rho u(e+\frac{u^2}{2})+pu)_x=\kappa\theta_{xx},\\
\end{aligned}
\right.
\end{equation}
where $x\in\mathbb{R}_+,t>0$ and
 $\rho(t,x)>0,u(t,x),\theta(t,x)>0,e(t,x)>0$ and $p(t,x)>0$
 are density, fluid velocity, absolute temperature, internal energy, and pressure respectively, while $\kappa>0$ is the coefficient of the heat conduction. Here we study the ideal and polytropic fluids so that $p$ and $e$ are given by the state equations
\begin{equation}\label{1.2}
p=R\rho\theta=A\rho^\ga\exp(\frac{\ga-1}{R}s),\quad e=C_v\theta
\quad( C_v=\frac{R}{\ga-1}),
\end{equation}
where $s$ is the entropy, $\ga>1$ is the adiabatic exponent and A,R are both positive constants.
 The
solution of (\ref{1.1})  satisfies the following initial data and the far field states that
\begin{equation}\label{1.3}
\left\{
\begin{aligned}
&(\rho,u,\theta)(0,x)=(\rho_0,u_0,\theta_0)(x), \\
&(\rho,u,\theta)(t,+\infty)=(\rho_+,u_+,\theta_+)=:z_+,
\end{aligned}
\right.
\end{equation}
where $\inf_{x\in \mathbb{R}_+}(\rho_0,\theta_0)(x)>0$ and $\rho_+>0, u_+, \theta_+>0$ are given constants.

As far as we know, there are very few results on the well-posed
problem for (1.1) due to the complexity and nonlinearity. Almost all the results are related to the analysis of
the global in time stability of the viscous Riemann solutions. More precisely, if
the heat effect  is also neglected, the Riemann solution consists of elementary waves
such as shock waves, rarefaction waves and contact discontinuities, which are dilation
invariant solutions of the Riemann problem (Euler system):
\begin{equation}\label{1.4}
\left\{
\begin{aligned}
&\rho_t+(\rho u)_x=0,\\
&(\rho u)_t+(\rho u^2+p)_x=0,\\
&\{\rho (e +\frac{u^{2}}{2})\}_t+\{\rho u(e +\frac{u^{2}}{2})+pu\}_x=0.
\end{aligned}
\right.
\end{equation}
The sound speed and Mach number  are
\begin{equation}\label{1.5}
c_s(\theta):=\sqrt{\frac{\gamma p}{\rho}}=\sqrt{ \gamma R \theta},\quad  M(\rho,u,\theta):=\frac{|u|}{{ c_s(\theta)}}.
\end{equation}
Then the inviscid Euler system (\ref{1.4}) has three characteristic speeds,
 they are
\begin{eqnarray}\label{1.6}
\lambda_1=u-c_s(\theta),\quad \lambda_2=u,\quad \lambda_3=u+c_s(\theta).
\end{eqnarray}
The system (\ref{1.4}) is  a typical example of the hyperbolic conservation laws,
it is of great importance to study the corresponding viscous system, such as isentropic or non-isentropic case. There are many works on the large-time behavior of the solutions to the Cauchy problem of the compressible Navier-Stokes equations. We refer to (\cite{F-M},\cite{Huang-L-M}, \cite{Huang-Matsumura},\cite{Huang-Matsumura-Xin},
\cite{H-T-X},\cite{T.P.Liu-1},
\cite{N-T-Z},\cite{Wang-Zhao}) and some references therein.

  Many authors also studied the initial boundary value problem
for the viscous and heat-conductive gas, which is modelled by
\begin{equation}\label{1.7}
\left\{
\begin{aligned}
&\rho_t+(\rho u)_x = 0, \\
&(\rho u)_t+(\rho u^2+p)_x = \mu u_{xx}, \\
&(\rho (e+\frac{u^2}{2}))_t+(\rho u(e+\frac{u^2}{2})+pu)_x = \kappa\theta_{xx}+(\mu uu_x)_x,
\end{aligned}
\right.
\end{equation}
where $\mu>0$ stands for the coefficients of viscosity. For the system (\ref{1.7}), we divide the phase space into following regions:
\begin{eqnarray*}
 \begin{aligned}
& {\Omega}_{sub}^+:=\left\{(\rho,u,\theta);\ 0<u<{c}_s(\theta)\right\},\quad
{\Omega}_{sub}^-:=\left\{(\rho,u,\theta);\ -{c}_s(\theta)<u<0\right\},\\
&{\Omega}^+_{supper}:=\left\{(\rho,u,\theta);\ u>{c}_s(\theta)\right\},\quad
{\Omega}^-_{supper}:=\left\{(\rho,u,\theta);\ u<-{c}_s(\theta)\right\},\\
&{\Gamma}^\pm_{trans}:=\left\{(\rho,u,\theta);\ |u|={c}_s(\theta)\right\},\quad
{\Gamma}_{sub}^0:=\left\{(\rho,u,\theta);\ u=0\right\}.
 \end{aligned}
 \end{eqnarray*}
For the inflow problem of (\ref{1.7}), Huang-Li-Shi \cite{Huang-L-S} studied the asymptotic stability of
 boundary layer  and its superposition with $3-$rarefaction wave.  Nakamura-Nishibata\cite{Naka-Nishi5} proved the existence and stability of boundary layer solution of (\ref{1.7}) in half space. Qin-Wang (\cite{Qin-Wang},\cite{Qin-Wang-2011}) proved the stability of
 the combination of BL-solution, rarefaction wave and viscous contact wave. For other interesting works, we refer to (\cite{F-H-W-Z}, \cite{F-X-R-2018}, \cite{H-M-S-1},\cite{H-M-S},\cite{Huang-Qin},
\cite{F-Z}-\cite{Kawashima-Zhu},\cite{Matsumura-Mei}-\cite{Nakamura},\cite{Nakamura2},\cite{Qin},\cite{Wan-Wang-Zou}).

Therefore, there is a natural question that how about the asymptotic stability of the
composite wave consisting of the boundary layer  and  $3-$rarefaction wave for the initial boundary value problems of the
non-viscous system (\ref{1.1}). We will give a positive answer to this problem
in this paper.  To do this,  we should define  proper boundary conditions.
Thus, we change the system (\ref{1.1}) in an equivalent form as
\begin{equation}\label{1.8}
\left\{
\begin{aligned}
&\rho_t+u\rho_x+\rho u_x=0,\\
& u_t+uu_x+\frac{p}{\rho^2}\rho_x=-R\theta_x,\\
&C_v\theta_t -\frac{\kappa }{\rho}\theta_{xx}=-C_v u \theta_x-R\theta u_x,
\end{aligned}
\right.
\end{equation}
then the two eigenvalues of the hyperbolic part are
\begin{equation}\label{1.9}
\widetilde{\lambda}_1=u-\widetilde{c}_s(\theta), \quad \widetilde{\lambda}_2=u+\widetilde{c}_s(\theta),
\end{equation}
 where
 \begin{equation}\label{1.10}
\widetilde{c}_s(\theta):=\sqrt{\frac{ p}{\rho}}=\sqrt{  R \theta}.
\end{equation}
Denote
\begin{equation}\label{1.11}
M_+=\frac{|u_+|}{\sqrt{\gamma R\theta_+}},\quad \tilde{M}(\rho,u,\theta)=
\frac{|u|}{\sqrt{R\theta}},\quad\tilde{M}_+=\frac{|u_+|}{\sqrt{R\theta_+}}
\end{equation}
for clear expression later.

By \cite{MPPT}, the boundary conditions of (\ref{1.1}) depend on the sign of $\widetilde{\lambda}_1$ and  $\widetilde{\lambda}_2$.
We consider that the global solution of (\ref{1.1}) is  in a small neighborhood
$\Omega(z_+)$ of $z_+$,  such that $\widetilde{\lambda}_i(i=1,2)$ at the boundary $x=0$ keeps the same sign with  $\widetilde{\lambda}_i(i=1,2)$ at the far field $x=+\infty,$
which are determined by the right state $z_+.$
Hence, we divide the phase space into new  regions
\begin{eqnarray*}
\begin{aligned}
   & \widetilde{\Omega}^+_{sub}:=\left\{(\rho,u,\theta);\ 0<u<\widetilde{c}_s(\theta)\right\},\quad
     \widetilde{\Omega}^-_{sub}:=\left\{(\rho,u,\theta);\ -\widetilde{c}_s(\theta)<u<0\right\};\\[2mm]
     & \widetilde{\Omega}^+_{supper}:=\left\{(\rho,u,\theta);\ u>\widetilde{c}_s(\theta)\right\},\quad
     \widetilde{\Omega}^-_{supper}:=\left\{(\rho,u,\theta);\ u<-\widetilde{c}_s(\theta)\right\};\\[2mm]
    & \widetilde{\Gamma}^+_{trans}:=\left\{(\rho,u,\theta);\ u=\widetilde{c}_s(\theta)\right\},\quad
    \widetilde{\Gamma}^-_{trans}:=\left\{(\rho,u,\theta);\ u=-\widetilde{c}_s(\theta)\right\},\\[2mm]
    &\widetilde{\Gamma}^0_{sub}:=\left\{(\rho,u,\theta);\ u=0\right\}.
    \end{aligned}
   \end{eqnarray*}
Then the boundary conditions are listed as follows:\\
Case (1):\ If  $z_+=(\rho_+,u_+,\theta_+) \in  \widetilde{\Omega}^-_{supper}$,  in the neighborhood of $U(z_+)$, $\widetilde{\lambda}_1<0$,
$\widetilde{\lambda}_2<0$, the boundary condition of $(\ref{1.1})$ is
\begin{eqnarray}\label{1.12}
\theta(t,0)=\theta_-.
\end{eqnarray}
Case (2):\ If  $z_+=(\rho_+,u_+,\theta_+) \in  \widetilde{\Omega}_{sub}^-\bigcup\widetilde{\Omega}_{sub}^+
\bigcup \widetilde{\Gamma}_{sub}^0$,  in the neighborhood of $U(z_+)$, $\widetilde{\lambda}_1<0$,
$\widetilde{\lambda}_2>0$, the boundary condition of $(\ref{1.1})$ is
 \begin{eqnarray}\label{1.13}
u(t,0)=u_-,\quad \theta(t,0)=\theta_-.
\end{eqnarray}
Case (3):\ If  $z_+=(\rho_+,u_+,\theta_+) \in  \widetilde{\Omega}^+_{supper}$,  in the neighborhood of $U(z_+)$, $\widetilde{\lambda}_1>0$,
$\widetilde{\lambda}_2>0$, the boundary condition of $(\ref{1.1})$ is
 \begin{eqnarray}\label{1.14}
\rho(t,0)=\rho_-,\quad u(t,0)=u_-,\quad \theta(t,0)=\theta_-.
\end{eqnarray}

 \begin{figure}[htp]\label{figuer 1}
 \centering
\includegraphics[width=4in,height=2.5in]{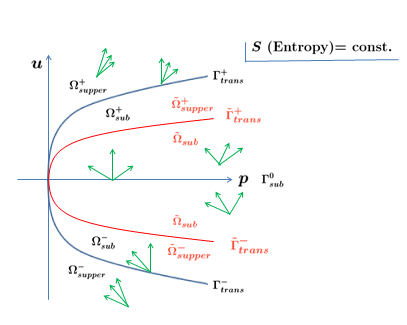}
\caption{ }
\end{figure}
Figure (1.1) comes from (\cite{MPPT}).

 Motivated by (\cite{Huang-L-S},\cite{Naka-Nishi5},\cite{Qin-Wang},\cite{Qin-Wang-2011}),  we are interested in studying
the inflow problem of   (\ref{1.1}),  (\ref{1.3}) and  (\ref{1.14}). we firstly discussed the existence of boundary layer solution
to system (\ref{1.1}) for $u_+>0$.
Precisely speaking,
if $(\rho_+,u_+,\theta_+)\in\tilde{\Omega}_{supper}^+\bigcap
\Omega_{sub}^+$, the  boundary layer solution is non-degenerate;
if  $(\rho_+,u_+,\theta_+)\in \Gamma_{trans}^+,$ the  boundary layer solution is degenerate.  Then we  proved the
unique global-in-time existence and the  asymptotic stability of both the boundary layer and  the superposition with
 the $3-$rarefaction wave  in supersonic case, that is, $u_+>\sqrt{R\theta_+}$, under some smallness conditions.
 We should mention that Nishibata and his group recently proved the existence and stability of boundary layer solution for a class of symmetric hyperbolic-parabolic systems, see \cite{Naka-Nishi3}. We occasionally know this excellent result by his lecture.  Our main analysis is  on the stability of combination of boundary layer solution and rarefaction wave,  which extended the result of \cite{Naka-Nishi3}.  There are also other interesting works for symmetric hyperbolic-parabolic system, see (\cite{Naka-Nishi4}-\cite{Naka-Nishi2}).

Our analysis is based on the  energy method. Since the fact that the non-viscous  system \eqref{1.1} is less dissipative,
 we need more subtle estimates to recover the regularity and dissipativity for the hyperbolic part.
 Precisely to say,
for cauchy problem of (\ref{1.1}),  \cite{F-M}
tell us that the perturbed solution  should be
in $C(H^2)$. However, in this paper,  the perturbed solution space (\ref{2.46}) implies that not only the  diameter derivatives need to be in
 $C(H^2),$  the  normal derivatives need to be in $C(H^1)$ specially.
 The second  main difficulty is how to control the higher order derivatives of boundary terms ( see $J(\tau,0)$  in Lemma 3.6).
 To do this, we should use the interior relations between functions  on the boundary, that's very helpful.
Moreover, some energy estimates on the normal direction besides the diameter direction should be needed.
  As far as we know, seldom works  use estimates on derivative of the normal direction
to study the asymptotic stability of the elementary waves. This method here maybe also helpful to other related
problems with similar analytical diffculties.

The present paper is organized as follows.  in section 2,
we obtain the existence of the boundary layer and some properties of the boundary layer and rarefaction wave,
 then we state our main results.
In section 3, we establish a priori estimates and  prove our main Theorem.\\

\noindent\emph{Notations.}
Throughout this paper, $c$ and $C$ denote some positive constants (generally large). $A\lesssim B$ means that there is a generic constant $C>0$ such that $A\leq CB$ and $A\sim B$ means $A\lesssim B$ and $B\lesssim A$.
For function spaces,~$L^p(\mathbb{R}_+)(1\leq p\leq \infty)$~denotes
the usual Lebesgue space on~$\mathbb{R}_+$~with norm~$\|{\cdot}\|_{L^p}$
and~$H^k(\mathbb{R}_+)$~the usual Sobolev space in the $L^2$ sense with norm~$\|\cdot\|_k$.
We note~$\|\cdot\|=\|\cdot\|_{L^2}$~for simplicity.
And $C^k(I; H^p)$ is the space of $k$-times continuously
differentiable functions on the interval $I$ with values in
$H^p(\mathbb{R}_+)$ and~$L^2(I; H^p)$ the space of~$L^2$-functions
on $I$ with values in~$H^p(\mathbb{R}_+)$.

\section{Boundary layer, Rarefaction wave and Main Results}
\setcounter{equation}{0}

\subsection{The existence of boundary layer}
In this section, we mainly discuss the existence of boundary layer solution
to system (\ref{1.1}) for $u_+>0$.
The boundary layer solution $(\bar{\rho},\bar{u},\bar{\theta})(x)$ to $(1.1)$ should satisfy
\begin{equation}\label{2.1}
\left\{
\begin{aligned}
&(\bar{\rho} \bar{u})_x=0,\\[2mm]
&(\bar{\rho}  \bar{u}^2+ \bar{p})_x=0,\\[2mm]
&(\bar{\rho} \bar{u}(\bar{e}+\frac{\bar{u}^2}{2})+\bar{p}\bar{u})_x=\kappa\bar{\theta}_{xx},\\[2mm]
\end{aligned}
\right.
\end{equation}
and
\begin{equation}\label{2.2}
\bar{\theta}(0)=\theta_-, \quad \quad \lim_{x\rightarrow+\infty}(\bar{\rho},\bar{u},\bar{\theta})(x)=(\rho_+,u_+,\theta_+).
\end{equation}
Integrating (\ref{2.1}) over $[x,+\infty)$, we have
\begin{equation}\label{2.3}
\left\{
\begin{aligned}
&\bar{\rho}\bar{u}=\rho_+ u_+,\\
&\bar{\rho}\bar{u}^2+ \bar{p}=\rho_+u_+^2+p_+,\\
&[\bar{\rho}\bar{u}(C_v\bar{\theta}+\frac{\bar{u}^2}{2})-\rho_+u_+(C_v\theta_++\frac{u_+^2}{2})]
+(\bar{p}\bar{u}-p_+u_+)=\kappa\bar{\theta}_{x},\\
\end{aligned}
\right.
\end{equation}
From $(\ref{2.3})_1$, we see that
\begin{equation*}
u_-=\frac{\rho_+u_+}{\bar{\rho}(0)}>0
\end{equation*}
is a  necessary condition.
Dividing both sides of $(\ref{2.3})_2$ by $\bar{\rho}\bar{u}(\rho_+u_+)$, we get
\begin{equation}\label{2.4}
\bar{u}-u_+=\frac{Ru_+(\theta_+-\bar{\theta})}{u_+\bar{u}-R\theta_+}.
\end{equation}
Substituing $(\ref{2.3})_1$, (\ref{2.4}) into $(\ref{2.3})_3$, system (\ref{2.3})
can be simplified as

\begin{equation}\label{2.5}
\kappa\bar{\theta}_{x}=R\rho_+u_+(\frac{(\bar{u}+u_+)u_+}{2(u_+\bar{u}-R\theta_+)}-\frac{\gamma}{\gamma-1})(\theta_+-\bar{\theta}).
\end{equation}

For convenience of discussion later, we  introducing

\begin{equation*}
 \tilde{w}_1(x)=\frac{\bar{u}(x)}{u_+}=\frac{\rho_+}{\bar{\rho}(x)}>0,\quad\quad\tilde{w}_2(x)=\frac{\bar{\theta}(x)}{\theta_+}>0.
\end{equation*}
Then (\ref{2.4}) tell us that
\begin{equation}\label{2.6}
u_+^2\tilde{w}_1^2(x)-(R\theta_++u_+^2)\tilde{w}_1(x)+R\theta_+\tilde{w}_2(x)=0.
\end{equation}

By (\ref{2.2}) $\lim_{x\rightarrow+\infty}\tilde{w}_1(x)=\lim_{x\rightarrow+\infty}\tilde{w}_2(x)=1$, therefore the relationship between $\tilde{w}_1(x)$ and $\tilde{w}_2(x)$ can be divided into two cases
from (\ref{2.6}),
\begin{equation}\label{2.7}
\left\{
\begin{aligned}
&\tilde{w}_1=\frac{(R\theta_++u_+^2)-\sqrt{(R\theta_++u_+^2)^2-4R\theta_+u_+^2\tilde{w}_2}}{2u_+^2}\quad\quad
0<u_+<\sqrt{R\theta_+},\\
&\tilde{w}_1=\frac{(R\theta_++u_+^2)+\sqrt{(R\theta_++u_+^2)^2-4R\theta_+u_+^2\tilde{w}_2}}{2u_+^2}\quad\quad
\sqrt{R\theta_+}<u_+.
\end{aligned}
\right.
\end{equation}
Above equation implies that $\tilde{w}_2(x)$ should satisfy
\begin{equation}\label{2.8}
\tilde{w}_2(x)\leq \frac{(R\theta_++u_+^2)^2}{4R\theta_+u_+^2}=\tilde{w}_{2\sup},
\end{equation}
where $\tilde{w}_{2\sup}\geq 1$, and $\tilde{w}_{2\sup}=1$ if and only if $u_+^2=R\theta_+$.
By the definitions of $\tilde{w}_1(x),\tilde{w}_2(x)$ and the relationship (\ref{2.7}),
again the symplified system (\ref{2.5}) becomes

\begin{equation}\label{2.9}
\kappa\tilde{w}_{2x}=L(\tilde{w}_2)(1-\tilde{w}_2)=H(\tilde{w}_2),
\end{equation}
where $u_+\gtrless\sqrt{R\theta_+}$,

\begin{equation}\label{2.10}
\begin{aligned}
L(\tilde{w}_2)=\frac{R\rho_+u_+}{2}(\frac{R\theta_++3u_+^2\pm\sqrt{(R\theta_++u_+^2)^2-4R\theta_+u_+^2\tilde{w}_2}}
{(u_+^2-R\theta_+)\pm\sqrt{(R\theta_++u_+^2)^2-4R\theta_+u_+^2\tilde{w}_2}}-\frac{2\gamma}{\gamma-1}),\\
\end{aligned}
\end{equation}
and the boundary condition of (\ref{2.9}) is derived from (\ref{2.2}),
\begin{equation}\label{2.10-1}
\tilde{w}_2(0)=\frac{\theta_-}{\theta_+},\quad\quad \tilde{w}_2(+\infty)=1.
\end{equation}
It is easy to check that when $u_+\gtrless \sqrt{R\theta_+}$
\begin{equation}\label{2.11}
\begin{aligned}
\frac{d}{d\tilde{w}_2}H(\tilde{w}_2)&=\frac{R\rho_+u_+}{2}
\{\frac{\mp G_2'(\tilde{w}_2)(2R\theta_++2u_+^2)}{((u_+^2-R\theta_+)\pm G_2(\tilde{w}_2))^2}(1-\tilde{w}_2)\\
&-(\frac{R\theta_++3u_+^2\pm G_2(\tilde{w}_2)}{(u_+^2-R\theta_+)\pm G_2(\tilde{w}_2)}-\frac{2\gamma}{\gamma-1})\}\\
\frac{d^2}{d^2\tilde{w}_2}H(\tilde{w}_2)=&\frac{R\rho_+u_+(2R\theta_++2u_+^2)G_2'(\tilde{w}_2)}{G_2(\tilde{w}_2)^2[(u_+^2-R\theta_+)\pm G_2(\tilde{w_2})]^3}\{\pm(u_+^2-R\theta_+)^3\\
&\pm3R\theta_+u_+^2(u_+^2-R\theta_+)(1-\tilde{w}_2)
+G_2(\tilde{w}_2)(u_+^2-R\theta_+)^2\}
\end{aligned}
\end{equation}
Where
\begin{equation}\label{2.12}
\begin{aligned}
&G_2(\tilde{w}_2)=\sqrt{(R\theta_++u_+^2)^2-4R\theta_+u_+^2\tilde{w}_2}\geq 0,\\
&\frac{dG_2(\tilde{w}_2)}{d\tilde{w}_2}=-\frac{2R\theta_+u_+^2}{G_2(\tilde{w}_2)}< 0.\\
\end{aligned}
\end{equation}
Hence the existence of the boundary layer solution to (\ref{2.1})-(\ref{2.2})
is equivalent to (\ref{2.9})-(\ref{2.10-1}). Now we start to study the latter.

Through our definition of $\tilde{w}_{2}(x)$ and (\ref{2.8}), it is obvious that the region of $\tilde{w}_2(x)$ for which the
boundary layer solution maybe exists should be $(0,\tilde{w}_{2sup}]$ and all the cases we considered below is under this premise. Here
we have known that $\tilde{w}_2=1$ is a solution of $H(\tilde{w}_2)=0$. If it has another solution $\tilde{w}_{2\ast}$, from (\ref{2.9}), it should satisfy
\begin{equation}\label{2.13}
\begin{aligned}
\frac{R\theta_++3u_+^2\pm G_2(\tilde{w}_{2\ast})}{(u_+^2-R\theta_+)\pm G_2(\tilde{w}_{2\ast})}
-\frac{2\gamma}{\gamma-1}=0, \quad\quad \text{if} \ u_+\gtrless\sqrt{R\theta_+}.
\end{aligned}
\end{equation}

Besides that, we denote the zero point of $\frac{d}{d\tilde{w}_2}H(\tilde{w}_2)$ by $\tilde{w}_{20},$
from $(\ref{2.11})_1$ ,  when $u_+\gtrless\sqrt{R\theta_+},$
\begin{equation}\label{2.14}
\begin{aligned}
&\frac{\mp G_2'(\tilde{w}_{20})(2R\theta_++2u_+^2)}
{((u_+^2-R\theta_+)+G_2(\tilde{w}_{20}))^2}(1-\tilde{w}_{20})
-(\frac{R\theta_++3u_+^2\pm G_2(\tilde{w}_{20})}{(u_+^2-R\theta_+)\pm G_2(\tilde{w}_{20})}
-\frac{2\gamma}{\gamma-1})=0
\end{aligned}
\end{equation}

We get following cases
\begin{itemize}
  \item (1)  If $u_+\in\tilde{\Omega}_{sub}^+$, that is
  $0<u_+<\sqrt{R\theta_+}(\tilde{M}_+<1)$. In this case,
  $\frac{d^2}{d^2\tilde{w}_2}H(\tilde{w}_2)>0$ for $\tilde{w}_2\in(0,\tilde{w}_{2sup}]$ from $(\ref{2.11})_2$. And
  \begin{equation}\label{2.14-1}
  \frac{d}{d\tilde{w}_2}H(\tilde{w}_2=1)>0.
  \end{equation}
  Then the convexity of $H(\tilde{w}_2)$ with $H(\tilde{w}_2=1)=0$ and (\ref{2.14-1}) tell us that there exists
  a small positive constant $\sigma.$ If $\tilde{w}_2(x)\in[1-\sigma,1], H(\tilde{w}_2)\leq 0,$ i.e., $\tilde{w}_{2x}\leq 0$ due to (\ref{2.9}). Thus, $\tilde{w}_2(x)$ is decreasing in $[1-\sigma,1].$ So when $\tilde{w}_2(0)<1,$
  $\tilde{w}_2(x)$ can not approach to $1$ as $x\rightarrow+\infty$. If $\tilde{w}_2(x)\in[1,1+\sigma], H(\tilde{w}_2)\geq 0,$ i.e., $\tilde{w}_{2x}\geq 0$. Thus, $\tilde{w}_2(x)$ is increasing in $[1,1+\sigma]$.
  Therefore when $\tilde{w}_2(0)>1,$ $\tilde{w}_2(x)$ can not approach to $1$ as $x\rightarrow+\infty.$ Consequently, there
  does not exist a solution to (\ref{2.9})-(\ref{2.10-1}) in this case.

 \item (2)If $u_+\in\tilde{\Omega}_{supper}^+\bigcap{\Omega}_{sub}^{+},$ that is $\sqrt{R\gamma\theta_+}>u_+>\sqrt{R\theta_+}(\tilde{M}_+>1$ and $M_+<1)$ .  In this case, $\frac{d^2}{d^2\tilde{w}_2}H(\tilde{w}_2)<0$ for $\tilde{w}_2\in(0,\tilde{w}_{2sup}]$ from $(\ref{2.11})_2$ and
 \begin{equation}\label{2.14-2}
 \begin{aligned}
 \lim_{\tilde{w}_2\rightarrow 0}H(\tilde{w}_2)=\frac{R\rho_+u_+}{2}(\frac{R\theta_++2u_+^2}
 {u_+^2}-\frac{2\gamma}{\gamma-1}).
 \end{aligned}
 \end{equation}
 There are two subcases.

 (2.1) When $(1<\gamma\leq 3)$ or $(\gamma>3$ and $u_+^2>\frac{\gamma-1}{2}R\theta_+),$
 \begin{equation}\label{2.14-3}
 \frac{d}{d\tilde{w}_2}H(\tilde{w}_2=1)<0,\quad\quad
 \lim_{\tilde{w}_2\rightarrow 0}H(\tilde{w}_2)<0.
 \end{equation}
 Combining the concavity of $H(\tilde{w}_2)$ with
 $H(\tilde{w}_2=1)=0$ and (\ref{2.14-3}), it tell us $\exists!\tilde{w}_{2\ast}\in(0,1)$
 such that (\ref{2.13}) holds in this subcase. Moreover, we could get $ 0<\tilde{w}_{2\ast}<\tilde{w}_{20}<1$.
 If $\tilde{w}_{2}(x)\leq\tilde{w}_{2\ast},$ then $H(\tilde{w}_2)\leq 0,$ i.e., $\tilde{w}_{2x}\leq 0$. That is, $\tilde{w}_{2}(x)$ is decreasing. So when  $\tilde{w}_{2}(0)\leq\tilde{w}_{2\ast},$
 $\tilde{w}_{2}(x)$ can not approach to 1 as $x\rightarrow+\infty.$ Consequently, there does not exist the solution to (\ref{2.9})-(\ref{2.10-1}). If $\tilde{w}_{2\ast}<\tilde{w}_{2}(x)<1,$ then $H(\tilde{w}_2)>0,$ i.e., $\tilde{w}_{2x}>0$.
 That is, $\tilde{w}_2(x)$ is increasing. Hence, when $ \tilde{w}_{2\ast}<\tilde{w}_{2}(0)<1,$ there exists a monotonically increasing solution $\tilde{w}_{2}(x)$ to (\ref{2.9})-(\ref{2.10-1}). Lastly,
 if $1<\tilde{w}_{2}(x)\leq\tilde{w}_{2\sup},$ then $H(\tilde{w}_2)<0$, i.e., $\tilde{w}_{2x}<0.$  That is ,$\tilde{w}_2(x)$ is decreasing.
 Therefore when
  $1<\tilde{w}_{2}(0)\leq\tilde{w}_{2\sup},$ there exists a monotonically decreasing solution $\tilde{w}_{2}(x)$ to (\ref{2.9})-(\ref{2.10-1}).
 Thus, we have proved that there exists a solution if and only if $\tilde{w}_{2\ast}<\tilde{w}_{2}(0)\leq\tilde{w}_{2\sup}$. And the decay estimates of the solution are obtained from (\ref{2.9}),

 \begin{equation}\label{2.14-4}
 \begin{aligned}
  |\frac{d^n}{dx^n}(\tilde{w}_{2}(x)-1)|\leq& C|\tilde{w}_2(0)-1|e^{-c_0x} \\
  &\quad\quad\quad\text{for}\quad
  n=1,2,3,...,\quad c_0=\frac{L(1)}{\kappa},
 \end{aligned}
 \end{equation}
 where $L(1)=L(\tilde{w}_2)\mid_{\tilde{w}_2=1}>0.$

 (2.2) When $(\gamma>3$ and $u_+^2\leq\frac{\gamma-1}{2}R\theta_+),$
 \begin{equation}\label{2.14-5}
 \frac{d}{d\tilde{w}_2}H(\tilde{w}_2=1)<0,\quad\quad
 \lim_{\tilde{w}_2\rightarrow 0}H(\tilde{w}_2)\geq 0,
 \end{equation}
 and $\lim_{\tilde{w}_2\rightarrow 0}H(\tilde{w}_2)=0$ holds if and only if $u_+^2=\frac{\gamma-1}{2}R\theta_+$. Combining the concavity of $H(\tilde{w}_2)$ with
 $H(\tilde{w}_2=1)=0$ and (\ref{2.14-5}), it tell us that $\tilde{w}_{2\ast}=0$ when $u_+^2=\frac{\gamma-1}{2}R\theta_+$, otherwise, $\tilde{w}_{2\ast}$ does not exist. If
 $0<\tilde{w}_2(x)<1, H(\tilde{w}_2)>0,$ i.e., $\tilde{w}_{2x}>0$. Therefore, when $0<\tilde{w}_2(0)<1$, there exists a monotonically increasing solution $\tilde{w}_{2}(x)$ to (\ref{2.9})-(\ref{2.10-1}). If $1<\tilde{w}_2(x)\leq\tilde{w}_{2sup}, H(\tilde{w}_2)<0,$ i.e., $\tilde{w}_{2x}<0$. So when $1<\tilde{w}_2(0)\leq\tilde{w}_{2sup}$, there exists a monotonically decreasing solution $\tilde{w}_{2}(x)$ to (\ref{2.9})-(\ref{2.10-1}).
 Hence for any $\tilde{w}_2(0)\in(0,\tilde{w}_{2\sup}],$ the solution to (\ref{2.9})-(\ref{2.10-1}) exists in this subcase. Moreover, the decay estimates of the solution are same as (\ref{2.14-4}).

 \item (3) If $u_+\in\Gamma_{trans}^+$, that is $u_+=\sqrt{R\gamma\theta_+}(M_+=1).$ In this case, $\frac{d^2}{d^2\tilde{w}_2}H(\tilde{w}_2)<0$ for $\tilde{w}_2\in(0,\tilde{w}_{2sup}]$  from $(\ref{2.11})_2$ and $\tilde{w}_{20}=\tilde{w}_{2\ast}=1.$
If $\tilde{w}_{2}(x)\leq 1,$ then $H(\tilde{w}_{2})\leq 0$, i.e, $\tilde{w}_{2x}\leq 0.$ Same as above discussion, when $\tilde{w}_2(0)\leq 1$,
$\tilde{w}_{2}(x)$ can not tends to 1 as $x\rightarrow+\infty.$ There does not exists a solution to
(\ref{2.9})-(\ref{2.10-1}). If $\tilde{w}_{2}(x)>1,$ then $H(\tilde{w}_2)<0$, i.e., $\tilde{w}_{2x}<0$. Therefore when $1<\tilde{w}_{2}(0)\leq \tilde{w}_{2sup}$, there exists a monotonically decreasing solution $\tilde{w}_{2}(x)$ to (\ref{2.9})-(\ref{2.10-1}). So in this case, the solution exists only for $\tilde{w}_2(0)\in (1,\tilde{w}_{2sup}].$
Moreover, by (\ref{2.11}), $\frac{d}{d\tilde{w}_2}H(\tilde{w}_2=1)=0$ and $\frac{d^2}{d^2\tilde{w}_2}H(\tilde{w}_2)\leq\frac{d^2}{d^2\tilde{w}_2}H(\tilde{w}_2=1)<0$ for $\tilde{w}_2\in[1,\tilde{w}_2(0)].$ Hence, the decay estimates of the solution
are obtained from (\ref{2.9}),
\begin{equation}\label{2.17}
|\frac{d^n}{dx^n}(\tilde{w}_{2}(x)-1)|\lesssim \frac{(\tilde{w}_2(0)-1)^{n+1}}
{(1+\tilde{c}_0(\tilde{w}_2(0)-1)x)^{n+1}} \quad\text{for}\quad
n=1,2,3,...,
\end{equation}
where
\begin{equation}\label{2.18}
\tilde{c}_0=-\frac{\frac{d^2}{d^2\tilde{w}_2}H(1)}{2\kappa}>0.
\end{equation}

 \item (4) If $u_+\in\Omega_{supper}^{+},$ that is $u_+>\sqrt{R\gamma\theta_+} (M_+>1).$ In this case, $\frac{d^2}{d\tilde{w}_2^2}H(\tilde{w}_2)<0$ for $\tilde{w}_2\in(0,\tilde{w}_{2sup}]$ from $(\ref{2.11})_2$. And
 \begin{equation}\label{2.19}
 \frac{d}{d\tilde{w}_2}H(\tilde{w}_2=1)>0.
 \end{equation}
Similiar as (1), there exists a small positive constant $\sigma$ such that when $\tilde{w}_2(x)\in[1-\sigma,1], \tilde{w}_{2x}\leq 0$, when $\tilde{w}_2(x)\in[1,1+\sigma],
\tilde{w}_{2x}\geq 0$. Therefore, there does not exist a solution to (\ref{2.9})-(\ref{2.10-1}) in this case.
\end{itemize}

Summarizing $(1)-(4)$, we have the following existence theorem of BL solution.

\begin{Proposition}\label{t1}
For $\gamma\in(1,+\infty),$ the boundary value problem (\ref{2.9})-(\ref{2.10-1}) has a unique smooth solution $\tilde{w}_2(x)$ if and only if $\tilde{M}_+>1$ and $M_+\leq1$ .
Precisely to say, for $\tilde{M}_+>1$ and $M_+<1$, there are two subcase:

(\romannumeral1) If $1<\gamma\leq 3$ or $\gamma>3$ and $u_+^2>\frac{\gamma-1}{2}R\theta_+$, there exists a unique smooth solution to (\ref{2.9})-(\ref{2.10-1}) when $\tilde{w}_2(0)\in(\tilde{w}_{2\ast},\tilde{w}_{2\sup}].$ Moreover, if
$\tilde{w}_2(0)\lessgtr 1,$ then $\tilde{w}_{2x}\gtrless 0$.

(\romannumeral2) If
$\gamma>3$ and $u_+^2\leq\frac{\ga-1}{2}R\theta_+$,  there exists a unique smooth solution
to (\ref{2.9})-(\ref{2.10-1})  when $\tilde{w}_2\in(0,\tilde{w}_{2sup}].$ Moreover,
 if $\tilde{w}_2(0)\lessgtr 1,$ then $\tilde{w}_{2x}\gtrless 0$. \\
And the decay estimates of the solution to both (\romannumeral1)  and (\romannumeral2) satisfy
(\ref{2.14-4}).

For $M_+=1,$ there exists a unique decreasing solution to (\ref{2.9})-(\ref{2.10-1}) when
$\tilde{w}_2(0)\in (1,\tilde{w}_{2sup}].$ And the decay estimates of this solution satisfy
(\ref{2.17}).
\end{Proposition}

\subsection{The properties of boundary layer solution and main result}

In this section, we construct the boundary layer, rarefaction wave for the initial boundary value problem (\ref{1.1}),(\ref{1.3}) and (\ref{1.14}) and  then state our main results. At first, change the Euler coordinates into Lagrange coordinates
\begin{equation}\label{2.20}
\left\{
\begin{aligned}
&v_t-u_x=0, \quad\quad x>s_-t,t>0,\\[2mm]
&u_t+p_x=0,\\
&(C_v\theta+\frac{u^2}{2})_t+(pu)_x=k(\frac{\theta_x}{v})_x,\\[2mm]
&(v,u,\theta)(t,s_-t)=(v_-,u_-,\theta_-),\\[2mm]
&(v,u,\theta)(0,x)=(v_0,u_0,\theta_0)(x)\rightarrow (v_+,u_+,\theta_+)( x\rightarrow+\infty),
\end{aligned}
\right.
\end{equation}
where $v=\frac{1}{\rho}$ is the specific volume of gas,the pressure $p=\frac{R\theta}{v}$ and the moving boundary has a speed $s_-=-\frac{u_-}{v_-}$.
Furthermore, introduce new variables $\xi=x-s_-t,$ then (\ref{2.20}) turns to
\begin{equation}\label{2.21}
\left\{
\begin{aligned}
&v_t-s_-v_\xi-u_\xi=0,\quad\quad \xi>0,\quad t>0,\\
&u_t-s_-u_\xi+p_\xi=0,\\
&(C_v\theta+\frac{u^2}{2})_t-s_-(C_v\theta+\frac{u^2}{2})_\xi
+(pu)_\xi=k(\frac{\theta_\xi}{v})_\xi,\\
&(v,u,\theta)(t,0)=(v_-,u_-,\theta_-),\\
&(v,u,\theta)(0,\xi)=(v_0,u_0,\theta_0)(\xi)\rightarrow (v_+,u_+,\theta_+)(\xi\rightarrow+\infty).
\end{aligned}
\right.
\end{equation}

And in this new coordinates,  the boundary layer solution $\bar{z}:=(\bar{v},\bar{u},\bar{\theta})(\xi)(\xi=x-s_-t)$ satisfies
\begin{equation}\label{2.22}
\left\{
\begin{aligned}
&-s_-\bar{v}_\xi-\bar{u}_\xi=0, \\
&-s_-\bar{u}_\xi+\bar{p}_\xi=0, \\
&-s_-(C_v\bar{\theta}+\frac{\bar{u}^2}{2})_\xi+(\bar{p} \bar{u})_\xi
=k(\frac{\bar{\theta}_\xi}{\bar{v}})_\xi, \\
&(\bar{v},\bar{u},\bar{\theta})(0)=(v_-,u_-,\theta_-), \quad u_->0, \\
&(\bar{v},\bar{u},\bar{\theta})(+\infty)=(v_+,u_+,\theta_+). \\
\end{aligned}
\right.
\end{equation}

Denote the strength of boundary layer solution as
\begin{equation}\label{2.23}
\bar{\delta}=|\theta_+-\theta_-|.
\end{equation}
For each $z_-=(v_-,u_-,\theta_-),$ we consider the situation of
$z_+\in\Om_{sub}^+\cap\tilde{\Om}_{supper}^+$ and  $z_-$ is located in a small
neighborhood of $z_+$. The neighborhood of $z_+$ denoted by $\Om_+$ later is given by
\begin{equation}\label{2.24}
\begin{aligned}
\Om_+=\{(v,u,\theta)|(v-v_+,u-u_+,\theta-\theta_+)|\leq \delta\}\subset\Om_{sub}^+\cap\tilde{\Om}_{supper}^+
\end{aligned}
\end{equation}
where $\delta$ is a positive constant depending only on $z_+$.

Then by the analysis in Section 2.1, we get the following lemma.

\begin{Lemma}\label{lemma1}(Property of boundary layer)$(\bar{v},\bar{u},\bar{\theta})$ satisfies

(1): If $z_+\in \tilde{\Om}_{supper}^+\bigcap{\Om}_{sub}^{+},$ that is $\sqrt{R\theta_+}<u_+<\sqrt{\ga R\theta_+}$, $\exists \bar{\delta}_0>0$, such that if $\bar{\delta}\lesssim \bar{\delta}_0,$ there exists a unique solution $(\bar{v}_x\lessgtr 0,\bar{u}_x\lessgtr 0,\bar{\theta}_x\gtrless 0)$ for (\ref{2.22})  which  is non-degenerate and satisfies
\begin{equation}\label{2.24-1}
|\frac{d^n}{d\xi^n}(\bar{v}-v_+,\bar{u}-u_+,\bar{\theta}-\theta_+)(\xi)|\lesssim \bar{\delta} e^{-c_0\xi}, \quad
n=0,1,2,3,....
\end{equation}

(2): If $z_+\in \Gamma_{trans}^+,$ that is $u_+=\sqrt{\ga R\theta_+},$ $\exists \bar{\delta}_1$ such that  if $0<\bar{\delta}\lesssim \bar{\delta}_1$, there exists a unique solution ($\bar{v}_x>0,\bar{u}_x>0,\bar{\theta}_x<0$) for
(\ref{2.22}) which is degenerate  and satisfies
\begin{equation}\label{2.25}
|\frac{d^n}{d\xi^n}(\bar{v}-v_+,\bar{u}-u_+,\bar{\theta}-\theta_+)(\xi)|\lesssim
\frac{\bar{\delta}^{n+1}}{1+(\bar{\delta}\xi)^{n+1}}.\quad n=0,1,2,3,....
\end{equation}
\end{Lemma}

This Lemma could be obtained immediately from our system (\ref{2.22})  and
Proposition \ref{t1}. In the following text, our discussion will take place in the
$ \tilde{\Om}_{supper}^+\bigcap{\Om}_{sub}^{+},$  and the
boundary layer is non-degenerate. To do so, we define the solution space as:
\begin{equation}\label{2.46}
 	\begin{aligned}
 	 \mathbb{X}_{\frac{1}{4}v_-,\frac{1}{4}\theta_-,M}(0,t):=&\{(\phi,\psi,\zeta)\in C([0,t];H^2(\mathbb{R}_+)),(\phi,\psi,\zeta)_t\in C([0,t];H^1(\mathbb{R}_+)),  \\
 	&(\phi,\psi)_\xi\in L^2(0,t;H^1(\mathbb{R}_+)),(\zeta_\xi,\zeta_t)\in L^2(0,t;H^2(\mathbb{R}_+)), \\
 	&\inf_{[0,t]\times\mathbb{R}_+}v(t,\xi)\geq\frac{1}{4}v_-
 	\inf_{[0,t]\times\mathbb{R}_+}\theta(t,\xi)\geq\frac{1}{4}\theta_-, \\
 	&\sup_{[0,t]\times\mathbb{R}_+}\|(\phi,\psi,\zeta)\|_2+\|(\phi_t,\psi_t,\zeta_t)\|_1\leq M	 \}.
 	\end{aligned}
 	\end{equation}

Then our first main result is as follow:
 \begin{Theorem} Assume that $z_-\in BL(z_+)\cap\Om_+$ and $z_+\in \Om_{sub}^+\cap\tilde{\Om}_{supper}^+$, that is,
$R\theta_-<u_-^2<\gamma R\theta_-(\gamma>1)$,  then
there exist some small positive constants $\delta_1$
 and $ \eta_1$  such that if $\bar{\delta}\lesssim \delta_1$ and

\begin{equation}\label{2.26}
\|(\phi_0,\psi_0,\zeta_0)\|_2+\|(\phi_t,\psi_t,\zeta_t)(0)\|_1\lesssim \eta_1,
\end{equation}
 the inflow problem (\ref{2.21}) has a unique solution $(v,u,\theta)(t,\xi)$
satisfying
\begin{equation}\label{2.27}
(v-\overline{v},u-\overline{u},\theta-\overline{\theta})(t,\xi)\in\mathbb{X}_{\frac{1}{4}v_-,\frac{1}{4}\theta_-,C_1(\eta_1+\bar{\delta})^{\frac{1}{2}}
}(0,+\infty)
\end{equation}
for some positive constant $C_1$. Furthermore, it holds that
\begin{equation}\label{2.28}
\sup_{\xi\geq 0}|(v,u,\theta)(t,\xi)-(\overline{v},\overline{u},\overline{\theta})(t,\xi)|\rightarrow 0, \ as
\  t\rightarrow +\infty.
\end{equation}
\end{Theorem}

\subsection{Rarefaction wave}
If $z_+\in R_3(z_-),$ that is, the 3-rarefaction wave $(v^r,u^r,\theta^r)(\frac{x}{t})$
connecting $z_-$ and $z_+$ is the unique weak solution globally in time
to the following Riemann problem:
\begin{equation}\label{2.30}
\left\{
\begin{aligned}
&v_t^r-u_x^r=0, \\
&u_t^r+p_x^r=0,  \\
&(e^r+\frac{u^{r^2}}{2})_t+(p^ru^r)_x=0, \\
&(v^r,u^r,\theta^r)(t,0)=
\begin{cases}
&(v_-,u_-,\theta_-),\quad x<0, \\
&(v_+,u_+,\theta_+),\quad x>0.
\end{cases}
\end{aligned}
\right.
\end{equation}
Here $\theta_-<\theta_+$ and $0<u_-<u_+$. To give the details of the large time behavior of the solutions to the inflow problem (\ref{2.21}), it is necessary to construct a smooth approximation $\tilde{z}:
=(\tilde{v},\tilde{u},\tilde{\theta})(t,x)$ of $(v^r,u^r,\theta^r)(\frac{x}{t})$. As in \cite{H-M-S-2}, consider
the solution to the following Cauchy problem:

\begin{equation}\label{2.31}
\left\{
\begin{aligned}
&w_t+ww_x=0, \\
&w(0,x)=
\begin{cases}
w_-,\quad x<0, \\
w_-+C_q\delta^r\int_0^{\eps x}y^q e^{-y}dy, x>0.
\end{cases}
\end{aligned}
\right.
\end{equation}
Here $\delta^r=w_+-w_->0, q>16$ are two constants, $C_q$ is a constant such that
$C_q\int_0^{+\infty}y^q e^{-y}dy=1,0<\eps<1$ is a small constant which will be determined later. Let
$w_\pm=\lambda_3(v_\pm,u_\pm,\theta_\pm),$ we construct the approximated function $\tilde{z}(t,x)$ by
\begin{equation}\label{2.32}
\left\{
\begin{aligned}
&S^r(\tilde{v},\tilde{u},\tilde{\theta})(t,x)=S^r(v_+,u_+,\theta_+),\\[2mm]
&\lambda_3(\tilde{v},\tilde{u},\tilde{\theta})(t,x)=w(1+t,x),\\[2mm]
&\tilde{u}=u_+-\int_{v_+}^{\tilde{v}}\lambda_3(s,S^r)ds.
\end{aligned}
\right.
\end{equation}
Remind that $\xi=x-s_-t,\tilde{z}(t,\xi)$ satisfy
\begin{equation}\label{2.33}
\left\{
\begin{aligned}
&\tilde{v}_t-s_-\tilde{v}_\xi-\tilde{u}_\xi=0,\\[2mm]
&\tilde{u}_t-s_-\tilde{u}_\xi+\tilde{p}_\xi=0,\\
&(\tilde{e}+\frac{\tilde{u}^2}{2})_t-s_-(\tilde{e}+\frac{\tilde{u}^2}{2})_\xi+(\tilde{p}\tilde{u})_\xi=0,\\
&(\tilde{v},\tilde{u},\tilde{\theta})(t,0)=(v_-,u_-,\theta_-), \quad u_->0, \\
&(\tilde{v},\tilde{u},\tilde{\theta})(0,\xi)=(\tilde{v}_0,\tilde{u}_0,\tilde{\theta}_0)(\xi)\rightarrow
(v_+,u_+,\theta_+)(\xi\rightarrow +\infty).
\end{aligned}
\right.
\end{equation}

\begin{Lemma}\label{lemma2}(Smooth rarefaction wave)(\cite{H-M-S-2})$\tilde{z}(t,\xi)$ satisfies

(1)$\tilde{u}_\xi\geq 0$, $\xi>0,t>0$.

(2)For any p($1\leq p\leq +\infty$), there exists a constant C such that
\begin{equation}\label{2.34}
\begin{aligned}
&\|(\tilde{v}_\xi,\tilde{u}_\xi,\tilde{\theta}_\xi)(t)\|_{L^p}\leq C_{pq}\min\{\delta^r\eps^{1-\frac{1}{p}},
(\delta^r)^\frac{1}{p}(1+t)^{-1+\frac{1}{p}}\}, \\
&\|(\tilde{v}_{\xi\xi},\tilde{u}_{\xi\xi},\tilde{\theta}_{\xi\xi})(t)\|_{L^p} \leq C_{pq}\min
\{\delta^r\eps^{2-\frac{1}{p}},((\delta^r)^{\frac{1}{p}}+(\delta^r)^{\frac{1}{q}})(1+t)^{-1+\frac{1}{q}}\},
\end{aligned}
\end{equation}

 (3)If $\xi+s_-t\leq \lambda_3(v_-,u_-,\theta_-)(1+t),$ then
$(\tilde{v},\tilde{u},\tilde{\theta})(t,\xi)\equiv(v_-,u_-,\theta_-).$

(4)$\lim_{t\rightarrow+\infty}\sup_{\xi\in R_+}|(\tilde{v},\tilde{u},\tilde{\theta})(t,\xi)-(v^r,u^r,\theta^r)(t,\xi)|=0.$
\end{Lemma}

Then our second main result is as follow:
 \begin{Theorem} Assume that $z_-\in R_3(z_+)\cap\Om_+$ and $z_+\in \Om_{sub}^+\cap\tilde{\Om}_{supper}^+$, that is,
$R\theta_-<u_-^2<\gamma R\theta_-(\gamma>1)$,
there exist some small positive constants $\delta_2$ and $ \eta_2$  such that if
 $ \epsilon\lesssim \delta_2$ and
\begin{equation}\label{2.36}
\|(\phi_0,\psi_0,\zeta_0)\|_2+\|(\phi_t,\psi_t,\zeta_t)(0)\|_1 \lesssim \eta_2,
\end{equation}
then the inflow problem (\ref{2.21}) has a unique solution $(v,u,\theta)(t,\xi)$
satisfying

\begin{equation}\label{2.37}
(v-v^r,u-u^r,\theta-\theta^r)(t,\xi)\in\mathbb{X}_{\frac{1}{4}v_-,\frac{1}{4}\theta_-,
C_2(\eta_2+\eps^{\frac{1}{8}})^{\frac{1}{2}}}(0,+\infty)
\end{equation}
for some positive constant $C_2$.
Furthermore, it holds that
\begin{equation}\label{2.38}
\sup_{\xi\geq 0}|(v,u,\theta)(t,\xi)-(v^r,u^r,\theta^r)(t,\xi)|\rightarrow 0, \ as
\  t\rightarrow +\infty.
\end{equation}
\end{Theorem}

\begin{Remark}
Note that the strength of rarefaction wave $\delta^r$ can not be suitably small in Theorem 2.2.
\end{Remark}

\subsection{Composition Waves}
For the left state $z_-\in BLR_3(z_+)\cap\Om_+$, we know that, there exists a unique point
$z_m:=(v_m,u_m,\theta_m)\in R_3(z_+)$ such that
 the BL-solution and the 3-rarefaction wave are connected by $z_m$.
Instead $z_-$ by $z_m$ in (\ref{2.33}), it holds that
\begin{equation}\label{2.41}
u_m=u_+-\int_{v_+}^{v_m}\lambda_3(\eta,S^r)d\eta.
\end{equation}
 For this $z_m$, instead $z_+$ by $z_m$ in  (\ref{2.22}),
we expect that the superposition of this boundary layer and the 3-rarefaction wave is stable. To do this, let
\begin{equation}\label{2.42}
(\hat{v},\hat{u},\hat{\theta})(t,\xi)=(\bar{v},\bar{u},\bar{\theta})(\xi)
+(\tilde{v},\tilde{u},\tilde{\theta})(t,\xi)-(v_m,u_m,\theta_m),
\end{equation}
and satisfies
\begin{equation}\label{2.43}
\left\{
\begin{aligned}
&\hat{v}_t-s_-\hat{v}_\xi-\hat{u}_\xi=0, \quad\xi>0,\quad t>0\\
&\hat{u}_t-s_-\hat{u}_\xi+\hat{p}_\xi=G_1,\\
&C_v\hat{\theta}_t-s_-C_v\hat{\theta}_\xi+\hat{p}\hat{u}_\xi
=k(\frac{\hat{\theta}_\xi}{\hat{v}})_\xi+G_2, \\
&(\hat{v},\hat{u},\hat{\theta})(t,0)=(v_-,u_-,\theta_-),\quad u_->0, \\
&(\hat{v},\hat{u},\hat{\theta})(0,\xi)=(\hat{v}_0,\hat{u}_0,\hat{\theta}_0)(\xi)\rightarrow (v_+,u_+,\theta_+),\quad \xi\rightarrow +\infty.
\end{aligned}
\right.
\end{equation}
where
\begin{equation}\label{2.44}
\begin{aligned}
G_1:&=(\hat{p}-\bar{p}-\tilde{p}+p_m)_\xi,\\
&=O(1)(|\bar{z}_\xi||\tilde{z}-z_m|+|\tilde{z}_\xi||\bar{z}-z_m|) \\
&=O(1)\bar{\delta} e^{-c(|\xi|+t)},\\[2mm]
G_2:&=(\hat{p}\hat{u}_\xi-\bar{p}\bar{u}_\xi-\tilde{p}\tilde{u}_\xi)-k(\frac{\hat{\theta}_\xi}{\hat{v}}
-\frac{\bar{\theta}_\xi}{\bar{v}})_\xi \\
&= O(1)(|\bar{z}_\xi||\tilde{z}-z_m|+|\tilde{z}_\xi||\bar{z}-z_m|)+O(1)(|\tilde{\theta}_{\xi\xi}|
+|\tilde{\theta}_\xi|^2) \\
&=O(1)\bar{\delta} e^{-c(|\xi|+t)}+O(1)(|\tilde{\theta}_{\xi\xi}|+|\tilde{\theta}_\xi|^2).
\end{aligned}
\end{equation}

Let
\begin{equation*}
(\phi,\psi,\zeta)=:(v-\hat{v},u-\hat{u},\theta-\hat{\theta})(t,\xi).
\end{equation*}
For $z_m\in\Om_+$, if $\bar{\delta}+\epsilon$ is small, then $z_-\in \Om_+$. Moreover, when  $\|(\phi_0,\psi_0,\zeta_0)\|_2+\|(\phi_t,\psi_t,\zeta_t)(0)\|_1$ is also small, we have
\begin{equation}\label{2.48}
\frac{3}{4} v_- \leq v_0(\xi)\leq \frac{5}{4}v_-,\quad\quad \frac{3}{4}\theta_-\leq\theta_0(\xi)\leq \frac{5}{4}\theta_-.
\end{equation}

The third main result is given below:
\begin{Theorem}\label{t3}Assume that $z_-\in BLR_3(z_+)\cap\Om_+$ and $z_+\in \Om_{sub}^+\cap\tilde{\Om}_{supper}^+$,
that is, $R\theta_-<u_-^2<\gamma R\theta_-(\ga>1)$,
There exist some small positive constants $\delta_3$ and $ \eta_3$,
such that if $\bar{\delta}+\epsilon \lesssim \delta_3$ and
\begin{equation}\label{2.49}
\|(\phi_0,\psi_0,\zeta_0)\|_2+\|(\phi_t,\psi_t,\zeta_t)(0)\|_1 \lesssim \eta_3,
\end{equation}
then the inflow problem (\ref{2.21}) has a unique solution $(v,u,\theta)(t,\xi)$
satisfying
\begin{equation}\label{2.50}
(v-\hat{v},u-\hat{u},\theta-\hat{\theta})(t,\xi)\in\mathbb{X}_{\frac{1}{4}v_-,\frac{1}{4}\theta_-,C_3(\eta_3+\bar{\delta}+\eps^{\frac{1}{8}})^{\frac{1}{2}}}(0,+\infty),
\end{equation}
for some positive constant $C_3$.
Furthermore, it holds that
\begin{equation}\label{2.51}
\sup_{\xi\geq 0}|(v,u,\theta)(t,\xi)-(\hat{v},\hat{u},\hat{\theta})(t,\xi)|\rightarrow 0, \quad as
\quad t\rightarrow +\infty.
\end{equation}
\end{Theorem}

\begin{Remark}
Note here the sterngth of boundary layer $\bar{\delta}$ should be so small,  the strength of rarefaction wave $\delta^r$ can not be suitably small.	
\end{Remark}

\section{Stability Analysis}
\setcounter{equation}{0}
In this section, we give the proofs of the main theorems. Since the results of Theorem 2.3 cover that of Theorem 2.1 and Theorem 2.2 if
$(v_\pm,u_\pm,\theta_\pm)=(v_m,u_m,\theta_m)$, we only show the asymptotic stability of the composition wave,  that is, Theorem 2.3.

\subsection{Reformed System}
Define  the perturbation function
\begin{equation}\label{3.1}
(\phi,\psi,\zeta)(t,\xi)=(v,u,\theta)(t,\xi)-(\hat{v},\hat{u},\hat{\theta})(t,\xi),
\end{equation}
then the reformed equation is
\begin{equation}\label{3.2}
\left\{
\begin{aligned}
&\phi_t-s_-\phi_\xi-\psi_\xi=0,\quad\xi>0, \quad t>0\\
&\psi_t-s_-\psi_\xi+(\frac{R\zeta}{v})_\xi-(\frac{\hat{p}\phi}{v})_\xi=-G_1,\\
&C_v\zeta_t-s_-C_v\zeta_\xi+p\psi_\xi+\hat{u}_\xi(p-\hat{p})
=\kappa(\frac{\zeta_\xi}{v}-\frac{\hat{\theta}_\xi\phi}{v\hat{v}})_\xi-G_2,\\
&(\phi,\psi,\zeta)(t,0)=(0,0,0), \\
&(\phi,\psi,\zeta)(0,\xi)=(\phi_0,\psi_0,\zeta_0)(\xi)\rightarrow (0,0,0),\quad as \quad \xi\rightarrow +\infty,
 \end{aligned}
\right.
\end{equation}
and the initial data satisfies the compatiable condition $(\phi_0,\psi_0,\zeta_0)(0)=(0,0,0).$

The local existence of  the solution to system (\ref{3.2}) is stated as follows :

\begin{Proposition}\label{p1}(Local existence) There exist positive constants $\overline{\delta}_1$,  $\bar{\eta}_1$
and $\overline{C},(\overline{C}\bar{\eta}_1\leq \eta_3)$ such that the following statements hold.
Under the assumption $\bar{\delta}+\epsilon\leq \overline{\delta}_1$, for any constant $M\in (0,\bar{\eta}_1)$, there exists a positive constant $t_0=t_0(M) $ not depending on $\tau$ such that if  $\|(\phi,\psi,\zeta)(\tau)\|_2+\|(\phi_t,\psi_t,\zeta_t)(\tau)\|_1\leq M$
and $\inf_{[0,t]\times\mathbb{R}_+}v(t,\xi)\geq \frac{1}{4}v_-,
 \inf_{[0,t]\times\mathbb{R}_+}\theta(t,\xi)\geq \frac{1}{4}\theta_-$, the problem  (\ref{3.2}) has a unique solution $(\phi,\psi,\zeta)(t,\xi)\in \mathbb{X}_{\frac{1}{4}v_-,\frac{1}{4}\theta_-,\bar{C}M}(\tau,\tau+t_0)$.
\end{Proposition}

\begin{proof} Consider system (\ref{3.2}) for any $\tau\geq 0$ in following forms:
	
	\begin{equation}\label{3.4-1}
	\left\{
	\begin{aligned}
	&\phi_t-s_-\phi_\xi-\psi_\xi=0,\quad \xi>0,\quad t>\tau,\\
	&\psi_t-s_-\psi_\xi-\frac{p}{v}\phi_\xi
	=\tilde{g}_1=g_1(\phi,\zeta,\zeta_\xi)-G_1,\\
	&C_v\zeta_t
	-\kappa\frac{\zeta_{\xi\xi}}{(\hat{v}+\phi)}=\tilde{g}_2:
	=g_2(\phi,\zeta,\phi_\xi,\psi_\xi,\zeta_\xi)-G_2,\\
	&\phi(t,0)=0, \psi(t,0)=0,\zeta(t,0)=0,\\
	&(\phi,\psi,\zeta)(\tau,\xi)=(\phi^{\tau},\psi^{\tau},\zeta^{\tau})(\xi)\rightarrow (0,0,0)(\xi\rightarrow +\infty),
	\end{aligned}
	\right.
	\end{equation}
	where
	\begin{equation}\label{3.4-2}
	\begin{aligned}
	&g_1(\phi,\zeta,\zeta_\xi)=\frac{\hat{p}_\xi\phi}{\hat{v}+\phi}
	-\frac{\hat{p}\hat{v}_\xi\phi}
	{(\hat{v}+\phi)^2}-\frac{R\zeta_\xi}{\hat{v}+\phi}+\frac{R\zeta\hat{v}_\xi}{(\hat{v}+\phi)^2}\\
	&g_2(\phi,\zeta,\phi_\xi,\psi_\xi,\zeta_\xi)=-\kappa\frac{\zeta_\xi(\hat{v}+\phi)_\xi}{(\hat{v}+\phi)^2}+s_-C_v\zeta_\xi
	-p\psi_\xi-\hat{u}_\xi(p-\hat{p})-\kappa(\frac{\hat{\theta}_\xi\phi}{v\hat{v}})_\xi
	\end{aligned}
	\end{equation}
Now we approximate $(\phi^{\tau},\psi^{\tau},\zeta^{\tau})(\xi)\in H^2$, by
	$(\phi^{\tau}_{j}, \psi^{\tau}_{j},\zeta^{\tau}_{j})\in H^m\cap H^2$
($m \geq 6$) such that
	\begin{equation}\label{3.4-3}
	(\phi^{\tau}_{j},\psi^{\tau}_{j},\zeta^{\tau}_{j})\rightarrow (\phi^{\tau},\psi^{\tau},\zeta^{\tau})\quad \text{strongly}\quad\quad \text{in}\quad\quad H^m
	\end{equation}
	as $j\rightarrow\infty$ and $\|(\phi^{\tau}_{j},\psi^{\tau}_{j},\zeta^{\tau}_{ j})\|_2+\|(\phi^{\tau}_{\tau j},\psi^{\tau}_{\tau j},\zeta^{\tau}_{\tau j})\|_1\leq CM(C>1).$  Moreover,
$$\inf_{\mathbb{R}_+}(\hat{v}+\phi^{\tau}_{j})(\tau,\xi)\geq \frac{1}{4}v_-, \quad \inf_{\mathbb{R}_+}(\hat{\theta}+\zeta^{\tau}_{j})(\tau,\xi)\geq \frac{1}{4} \theta_-$$
hold for any $j\geq 1.$
	
	We will use the iteration method to prove our Proposition 3.1.  Define the sequence
	$\{(\phi^{(n)}_j,\psi^{(n)}_j,\zeta^{(n)}_j)(t,\xi)\}$ for each $j$ so that
	
	\begin{equation}\label{3.4-4}
	(\phi^{(0)}_j,\psi^{(0)}_j,\zeta^{(0)}_j)(t,\xi)=(\phi^{\tau}_{j},\psi^{\tau}_{j},
	\zeta^{\tau}_{j})(\xi)
	\end{equation}
	and
	$(\phi^{(n)}_j,\psi^{(n)}_j,\zeta^{(n)}_j)(t,\xi)(n\geq1)$ is the solution to the following equation

	\begin{equation}\label{3.4-5}
	\left\{
	\begin{aligned}
	&\phi_{jt}^{(n)}-s_-\phi_{j\xi}^{(n)}-\psi_{j\xi}^{(n)}=0,\quad \xi>0\quad t>\tau,\\
	&\psi_{jt}^{(n)}-s_-\psi_{j\xi}^{(n)}-\frac{R(\zeta^{(n-1)}_j+\widehat{\theta})}{(\phi^{(n-1)}_j+\widehat{v})^2}\phi^{(n)}_{j\xi}
	=\tilde{g}_1^{(n-1)}(\phi^{(n-1)}_j,\zeta^{(n-1)}_j,\zeta^{(n-1)}_{j\xi}),\\
	&C_v\zeta_{jt}^{(n)}-\kappa\frac{\zeta_{j\xi\xi}^{(n)}}{(\hat{v}+\phi_{j}^{(n-1)})}
	=\tilde{g}_2^{(n-1)}(\phi^{(n-1)}_j,\zeta^{(n-1)}_j,\zeta^{(n-1)}_{j\xi}),\\
	&\phi^{(n)}_j(t,0)=0,\quad\psi^{(n)}_j(t,0)=0,\quad\zeta^{(n)}_j(t,0)=0,\\
	&(\phi^{(n)}_j,\psi^{(n)}_j,\zeta^{(n)}_j)(\tau,\xi)=(\phi^{\tau}_{j},\psi^{\tau}_{j},\zeta^{\tau}_{j})(\xi),
	\end{aligned}
	\right.
	\end{equation}
where
	\begin{equation}\label{3.4-6}
	\begin{aligned}
 \tilde{g}_1^{(n-1)}&=:g_1^{(n-1)}(\phi^{(n-1)}_j,\zeta^{(n-1)}_j,\zeta^{(n-1)}_{j\xi})-G_1\\
\tilde{g}_2^{(n-1)}&=:g_2^{(n-1)}(\phi_{j}^{(n-1)},\zeta_j^{(n-1)},\phi_{j\xi}^{(n-1)},\psi_{j\xi}^{(n-1)},\zeta_{j\xi}^{(n-1)})-G_2.
	\end{aligned}
	\end{equation}
  We now assume that $\bar{\eta}_1$ suitably small, if $\tilde{g}_2^{(n-1)}\in C(\tau,\tau+t_0;H^{m-1})$,
 $\zeta^{\tau}_{j}\in H^{m},$ then there exists a unique local solution $\zeta_{j}^{(n)}$ to (\ref{3.4-5}) satisfying
	\begin{equation}\label{3.4-7}
	\zeta^{(n)}_j\in C(\tau,\tau+t_0;H^m)\cap C^1(\tau,\tau+t_0;H^{m-2})\cap L^2(\tau,\tau+t_0;H^{m+1})
	\end{equation}
	
Making use of this, if $(\phi^{(n-1)}_j,\psi^{(n-1)}_j,\zeta^{(n-1)}_j)(t,\xi)
     \in\mathbb{X}_{\frac{1}{4}v_-,\frac{1}{4}\theta_-,\bar{C}M}(\tau,\tau+t_0),$
	from system (\ref{3.4-5}), by Gronwall inequality, we immediately get that for
	$t\in[\tau,\tau+t_0],$
	\begin{equation}\label{3.4-8}
	\begin{aligned}
	&\|\zeta^{(n)}_j(t)\|_2^2+\|\zeta^{(n)}_{jt}(t)\|_1^2+
	\int_\tau^{\tau+t_0}\|\zeta^{(n)}_{\xi}(s)\|_2^2+\|\zeta^{(n)}_t(s)\|_2^2ds\\
	&\leq e^{C(v_-,\theta_-,\bar{\eta}_1)t_0}(\|\zeta^{\tau}_j(\tau)\|_2^2+\|\zeta^\tau_{j\tau}(\tau)\|_1^2
	+C(v_-,\theta_-,\bar{\eta}_1)t_0),
	\end{aligned}
	\end{equation}
	Then a direct computation on $(\ref{3.4-5})_{1,2}$ with (\ref{3.4-8}) also tell us
	
	\begin{equation}\label{3.4-9}
	\begin{aligned}
	&\|(\phi^{(n)}_j,\psi^{(n)}_j)(t)\|_2^2+\|(\phi_{jt}^{(n)},\psi_{jt}^{(n)})(t)\|_1^2 \\[2mm]
	\leq & e^{C(v_-,\theta_-,\bar{\eta}_1)t_0}\{\|(\phi^{\tau}_j,\psi^{\tau}_j,\zeta^{\tau}_j)\|_2^2
	+\|(\phi^\tau_{j\tau},\psi^\tau_{j\tau},\zeta^\tau_{j\tau})\|_1^2\\
	&+C(v_-,\theta_-,\bar{\eta}_1)t_0
	+C(\int_\tau^{\tau+t_0}\|\zeta^{(n-1)}_{j\xi\xi\xi}(s)\|^2+\|\zeta^{(n-1)}_{jt\xi\xi}(s)\|^2ds)\}.
	\end{aligned}
	\end{equation}
	Combining (\ref{3.4-8}) and (\ref{3.4-9}), as long as $t_0$ suitably small, we finally get
	\begin{equation}\label{3.4-10}
	\|(\phi^{(n)}_j,\psi^{(n)}_j,\zeta^{(n)}_j)(t)\|_2+\|(\phi_{jt}^{(n)},\psi_{jt}^{(n)},\zeta_{jt}^{(n)})(t)\|_1\leq \bar{C}M, \quad t\in[\tau,\tau+t_0].
	\end{equation}
If $\bar{\eta}_1$ suitably small, by Sobolev's inequality
 the sequence $(\phi^{(n)}_j,\psi^{(n)}_j,\zeta^{(n)}_j)$ is uniformly bounded in the function space $\mathbb{X}_{\frac{1}{4}v_-,\frac{1}{4}\theta_-, \bar{C}M}(\tau,\tau+t_0).$ By using the same method in (\cite{Kawa}), we can  finally prove that $(\phi^{(n)}_j,\psi^{(n)}_j,\zeta^{(n)}_j)$ has a subsequence $(\phi^{(n')}_j,\psi^{(n')}_j,\zeta^{(n')}_j)\rightarrow (\phi_j,\psi_j,\zeta_j)\in \mathbb{X}_{\frac{1}{4}v_-,\frac{1}{4}\theta_-,\bar{C}M}(\tau,\tau+t_0)$  as $n'\rightarrow \infty$. Again, we let $j\rightarrow\infty,$ we can obtain the desired unique local solution $(\phi,\psi,\zeta)(t,\xi)\in
	\mathbb{X}_{\frac{1}{4}v_-,\frac{1}{4}\theta_-,\bar{C}M}(\tau,\tau+t_0)$ under the assumption $t_0$ is small enough. Thus Proposition \ref{p1} has been proved.

\end{proof}

 Set
\begin{eqnarray}\label{3.5}
\begin{aligned}
N(T):=\sup_{t\in[0,T]}\|(\phi,\psi,\zeta)(t)\|_2+\|(\phi_t,\psi_t,\zeta_t)(t)\|_1,
 \end{aligned}
\end{eqnarray}
Suppose that  $ (\phi,\psi,\zeta)(t,\xi) $  obtained in Proposition 3.1 has been extended to some time $T>t$,
 we want to get the following a priori estimates to obtain a global solution.

\begin{Proposition} (A priori estimates) \label{p2}Under the conditions listed in Theorem \ref{t3},
$(\phi,\psi,\zeta)(t,\xi) \in \mathbb{X}_{\frac{1}{4}v_-,\frac{1}{4}\theta_-,N(T)}(0,T)$ is the solution of
the problem  (\ref{3.2}) obtained in Proposition 3.1 which has been extended to some $T>0$, and there
exists  $\overline{\delta}_2$ and $\overline{\eta}_2$
 such that if $ \epsilon+\bar{\delta}\leq \overline{\delta}_2$  and $N(T)\leq \overline{\eta}_2$,
 then it holds that for $t\in[0,T]$,
\begin{eqnarray}\label{3.6}
   \begin{aligned}
&\|(\phi,\psi,\zeta)(t)\|^2_2+\|(\phi_t,\psi_t,\zeta_t)(t)\|^2_1+\int^t_0\|(\phi_\xi,\psi_\xi,\zeta_\xi,\zeta_{\xi\xi})(\tau)\|^2_1d\tau \\[2mm]
&+\int^t_0\|(\zeta_{t\xi},\zeta_{t\xi\xi})(\tau)\|^2d\tau
+\int^t_0(\psi_\xi^2+\phi_\xi^2+\zeta_\xi^2)(\tau,0)d\tau\\[2mm]
&+\int^t_0(\psi_{\xi\xi}^2+\phi_{\xi\xi}^2+\zeta_{\xi\xi}^2+\psi_{t\xi}^2+\phi_{t\xi}^2+\zeta_{t\xi}^2)(\tau,0)d\tau \\[2mm]
\lesssim & \|(\phi_0,\psi_0,\zeta_0)\|^2_2+\|(\phi_t,\psi_t,\zeta_t)(0)\|^2_1+\bar{\delta}+\epsilon^{\frac{1}{8}}.
\end{aligned}
\end{eqnarray}
\end{Proposition}

Once Proposition \ref{p2} is proved, we can extend the unique local solution $(\phi,\psi,\zeta)(t,\xi) $
 obtained in Proposition \ref{p1} to $t=\infty$, moreover, estimate (\ref{3.6}) implies that
\begin{eqnarray}\label{3.7}
  \int^\infty_0 \big( \| (\phi_\xi,\psi_\xi,\zeta_\xi)(t) \|^2 + \big| \frac{d}{dt} \| (\phi_\xi,\psi_\xi,\zeta_\xi)(t) \|^2 \big| \big) d \tau
 < + \infty \, ,
\end{eqnarray}
which together with  Sobolev inequality easily leads to the asymptotic
behavior  (\ref{2.51}), this  concludes the proof of Theorem 2.3.
In the rest of this section, our main task is to show the a priori estimates.

\subsection{A Priori Estimates}
In the following part of this section, we mainly proof the  Proposition 3.2, under the assumption $\bar{\delta}+\epsilon\leq \overline{\delta}_2$,
$N(t)\leq \overline{\eta}_2$, $\hat{v},v,\hat{\theta}, \theta$ are uniformly positive on $[0,T]$ by Sobolev's inequality as
\begin{eqnarray}\label{3.8}
   \begin{aligned}
   \inf_{x}\hat{v}\geq \frac{3v_-}{4},\ \inf_{x,t}v\geq \frac{v_-}{4}, \
\inf_{x}\hat{\theta}\geq \frac{3\theta_-}{4}, \ \inf_{x,t}\theta\geq \frac{\theta_-}{4}.
 \end{aligned}
\end{eqnarray}
which will be used later.
At first, we show the basic estimates.

\begin{Lemma} Under the same assumptions listed in Proposition 3.2, if $\bar{\delta},  \epsilon, N(T)$ are suitably small, it holds that
\begin{equation}\label{3.9}
\begin{aligned}
&\|(\phi,\psi,\zeta)(t)\|^2+\IT(\|\sqrt{\tilde{u}_\xi}(\phi,\zeta)(\tau)\|^2+\|\zeta_\xi(\tau)\|^2)d\tau \\
\lesssim & \|(\phi_0,\psi_0,\zeta_0)\|^2+(\overline{\delta}+\epsilon^{\frac{1}{8}})\{\IT\|(\phi_\xi,\psi_\xi)(\tau)\|^2d\tau+1\}.
\end{aligned}
\end{equation}
\end{Lemma}
\begin{proof} Define the energy form
\begin{equation}\label{3.10}
E=R\hat{\theta}\Phi(\frac{v}{\hat{v}})+\frac{\psi^2}{2}+C_v\hat{\theta}\Phi(\frac{\theta}{\hat{\theta}}),
\end{equation}
where $\Phi(s)=s-1-\ln s$. Obviously, there exists a positive constant C(s) such that

\begin{equation*}
C(s)^{-1}s^2\leq \Phi(s)\leq C(s)s^2,
\end{equation*}
we can get the following estimate
\begin{equation}\label{3.11}
\begin{aligned}
&E_t-s_-E_\xi+\kappa\frac{\hat{\theta}\zeta_\xi^2}{v\theta^2}
+\hat{p}\tilde{u}_\xi(\Phi(\frac{\theta\hat{v}}{v\hat{\theta}})+\ga\Phi(\frac{v}{\hat{v}})) \\
&+((p-\hat{p})\psi-\kappa(\frac{\zeta_\xi}{v}-\frac{\hat{\theta}_\xi\phi}
{v\hat{v}})\frac{\zeta}{\theta})_\xi
=G_3-G_1\psi-G_2\frac{\zeta}{\theta},
\end{aligned}
\end{equation}
where
\begin{equation}\label{3.12}
\begin{aligned}
G_3&=-\hat{p}\bar{u}_\xi(\Phi(\frac{\theta\hat{v}}{v\hat{\theta}})+\ga\Phi(\frac{v}{\hat{v}}))
+[\kappa(\frac{\hat{\theta}_\xi}{\hat{v}})_\xi+G_2][(\ga-1)\Phi(\frac{v}{\hat{v}})
+\Phi(\frac{\theta}{\hat{\theta}})-\frac{\zeta^2}{\theta\hat{\theta}}]\\
&+\kappa\frac{\hat{\theta}_\xi\zeta_\xi\zeta}{v\theta^2}+\kappa(\frac{1}{v}-\frac{1}{\hat{v}})\frac{\zeta\hat{\theta}_\xi^2}
{\theta^2}-\kappa(\frac{1}{v}-\frac{1}{\hat{v}})\frac{\hat{\theta}\hat{\theta}_\xi\zeta_\xi}{\theta^2}.
\end{aligned}
\end{equation}
It is easy to see that
\begin{equation}\label{3.13}
\begin{aligned}
|G_3|\lesssim \frac{1}{4}\frac{\kappa\hat{\theta}\zeta_\xi^2}{v\theta^2}
+|(\bar{u}_\xi,(\bar{\theta}-\theta_m)\tilde{\theta}_\xi+\bar{\theta}_\xi(\tilde{\theta}-\theta_m),
\tilde{\theta}_{\xi\xi}+\tilde{\theta}_\xi^2)|(\phi^2+\zeta^2),
\end{aligned}
\end{equation}
Since
\begin{equation}\label{3.14}
|f(\xi)|=|f(0)+\IX f_y dy|\leq|f(0)|+\sqrt{\xi}\|f_\xi\|,
\end{equation}
 and by the fact that $(\phi,\psi,\zeta)(t,0)=(0,0,0)$, we get
\begin{equation}\label{3.15}
\begin{aligned}
&\IT\int_{R_+}|\bar{u}_\xi|(\phi^2+\zeta^2)d\xi d\tau \\
\lesssim & \bar{\delta}\IT\IXX e^{-c\xi}(\phi^2+\zeta^2)(t,\xi)d\xi d\tau \\
\lesssim & \bar{\delta}\IT\IXX \xi e^{-c\xi}\|(\phi_\xi,\zeta_\xi)(\tau)\|^2d\xi d\tau \\
\lesssim & \bar{\delta}\IT\|(\phi_\xi,\zeta_\xi)(\tau)\|^2d\tau.
\end{aligned}
\end{equation}
By the properties of rarefaction wave as
\begin{equation}\label{3.16}
\begin{aligned}
&\|(\tilde{v}_\xi,\tilde{u}_\xi,\tilde{\theta}_\xi)(t)\|^2\lesssim \va^\frac{1}{8}(1+t)^{-\frac{7}{8}},\\[2mm]
&\|(\tilde{v}_{\xi\xi},\tilde{u}_{\xi\xi},\tilde{\theta}_{\xi\xi})(t)\|_{L^1}\lesssim \va^\frac{1}{8}(1+t)^{-\frac{13}{16}},
\end{aligned}
\end{equation}
we have
\begin{equation}\label{3.17}
\begin{aligned}
&\IT\IXX(|\tilde{\theta}_{\xi\xi}|+|\tilde{\theta}_\xi|^2)(\phi^2+\zeta^2)(t,\xi)d\xi d\tau \\[2mm]
\lesssim &\IT(\|\phi\|\|\phi_\xi\|+\|\zeta\|\|\zeta_\xi\|)(\|\tilde{\theta}_\xi\|^2
+\|\tilde{\theta}_{\xi\xi}\|_{L^1})d\xi d\tau \\[2mm]
\lesssim & \epsilon^{\frac{1}{8}}\IT (1+\tau)^{-\frac{13}{8}}\|(\phi,\zeta)(\tau)\|^2+\|(\phi_\xi,\zeta_\xi)(\tau)\|^2d\tau \\[2mm]
\lesssim & \epsilon^{\frac{1}{8}}\{1+\IT\|(\phi_\xi,\zeta_\xi)(\tau)\|^2d\tau\}.
\end{aligned}
\end{equation}
And the rest term  satisfy
\begin{equation}\label{3.18}
\begin{aligned}
&\IT\IXX((\bar{\theta}-\theta_m)\tilde{\theta}_\xi+\bar{\theta}_\xi(\tilde{\theta}-\theta_m))(\phi^2+\zeta^2)d\xi d\tau\\
\lesssim & \bar{\delta}\IT\IXX e^{-c(\xi+\tau)}(\phi^2+\zeta^2)d\xi d\tau \\
\lesssim & \bar{\delta}\IT(\|\phi\|\|\phi_\xi\|+\|\zeta\|\|\zeta_\xi\|)e^{-c\tau}d\tau \\
\lesssim & \bar{\delta}\{1+\IT\|(\phi_\xi,\zeta_\xi)(\tau)\|^2d\tau\}.
\end{aligned}
\end{equation}
Integrating (\ref{3.11}), and making use of the estimates (\ref{3.13})-(\ref{3.18}), we get
\begin{equation}\label{3.19}
\begin{aligned}
&\|(\phi,\psi,\zeta)(t)\|^2+\IT(\|\zeta_\xi(\tau)\|^2+\|\sqrt{\tilde{u}_\xi}(\phi,\zeta)(\tau)\|^2)d\tau \\
\lesssim &\|(\phi_0,\psi_0,\zeta_0)\|^2+(\bar{\delta}+\epsilon^{\frac{1}{8}})\{1+\IT\|(\phi_\xi,\zeta_\xi)(\tau)\|^2d\tau\}\\[2mm]
&+\IT\IXX|G_1\psi|+|G_2\zeta|d\xi d\tau,
\end{aligned}
\end{equation}
where
\begin{equation}\label{3.20}
\begin{aligned}
&\IT\IXX|G_1\psi|+|G_2\zeta|d\xi d\tau  \\
\lesssim & \IT\IXX\bar{\delta} e^{-c(\xi+\tau)}(\|\psi\|^{\frac{1}{2}}\|\psi_\xi\|^{\frac{1}{2}}+
\|\zeta\|^{\frac{1}{2}}\|\zeta_\xi\|^{\frac{1}{2}})d\xi d\tau\\[2mm]
&+\IT(\|\zeta\|^{\frac{1}{2}}\|\zeta_\xi\|^\frac{1}{2}(\|\tilde{\theta}_\xi\|^2+\|\tilde{\theta}_{\xi\xi}\|_{L^1})d\tau \\
\lesssim & \bar{\delta}\IT e^{-c\tau}\|(\psi,\zeta)(\tau)\|^{\frac{2}{3}}+\|(\psi_\xi,\zeta_\xi)(\tau)\|^2d\tau\\[2mm]
&+\IT \|\zeta(\tau)\|^{\frac{1}{2}}\|\zeta_\xi(\tau)\|^\frac{1}{2}[\epsilon^\frac{1}{8}(1+\tau)^{-\frac{7}{8}}
+\epsilon^\frac{1}{8}(1+\tau)^{-\frac{13}{16}}]d\tau \\
\lesssim & \bar{\delta}\IT e^{-c\tau}(\|(\psi,\zeta)(\tau)\|^2+1)d\tau
+(\bar{\delta}+\epsilon^\frac{1}{8})\IT\|(\psi_\xi,\zeta_\xi)(\tau)\|^2d\tau\\[2mm]
&+\epsilon^\frac{1}{8}\IT(1+\tau)^{-\frac{13}{8}}\|\zeta(\tau)\|^\frac{2}{3}d\tau \\
\lesssim &(\bar{\delta}+\epsilon^{\frac{1}{8}})\{1+\IT\|(\psi_\xi,\zeta_\xi)(\tau)\|^2d\tau\}.
\end{aligned}
\end{equation}
Inserting (\ref{3.20}) into (\ref{3.19}), we can get the estimate (\ref{3.9}) and complete the proof of Lemma 3.1.
\end{proof}

\begin{Lemma} Under the same assumptions listed in Proposition 3.2, if $\bar{\delta},  \epsilon, N(T)$ are suitably small,  then it holds that

\begin{equation}\label{3.21}
\begin{aligned}
&\|(\phi_\xi,\psi_\xi,\zeta_\xi)(t)\|^2+\int^t_0(\psi_\xi^2+\phi_\xi^2+\zeta_\xi^2)(\tau,0)d\tau
+\int^t_0\|\zeta_{\xi\xi}(\tau)\|^2d\tau \\[2mm]
\lesssim & \|(\phi_0,\psi_0,\zeta_0)\|_1^2+(\bar{\delta}+\epsilon^{\frac{1}{8}})
+(\bar{\delta}+\epsilon^{\frac{1}{8}}+N(T))\IT\|(\phi_\xi,\psi_\xi)(\tau)\|^2_1d\tau
\end{aligned}
\end{equation}

\end{Lemma}
\begin{proof}  Multiplying  $(\ref{3.2})_1$ by $-\frac{\hat{p}}{\hat{v}}\phi_{\xi\xi}$ and $(\ref{3.2})_2$ by $-\psi_{\xi\xi}$, $(\ref{3.2})_3$ by $-\frac{\zeta_{\xi\xi}}{\hat{\theta}}$ and adding the results, we can get
\begin{equation}\label{3.22}
\begin{aligned}
&\frac{1}{2}(\frac{\hat{p}}{\hat{v}}\phi_\xi^2+\psi_\xi^2+C_v\frac{\zeta_\xi^2}{\hat{\theta}})_t+\kappa\frac{\zeta_{\xi\xi}^2}{v\hat{\theta}}
+\frac{s_-}{2}\{\frac{\hat{p}}{\hat{v}}\phi_\xi^2+\psi_\xi^2+C_v\frac{\zeta_\xi^2}{\hat{\theta}}\}_\xi\\
&+\{\frac{\hat{p}}{\hat{v}}\phi_\xi\psi_\xi-\frac{R}{\hat{v}}\zeta_\xi\psi_\xi\}_\xi
-(\frac{\hat{p}}{\hat{v}}\phi_\xi\phi_t+\psi_\xi\psi_t+C_v\frac{\zeta_\xi}{\hat{\theta}}\zeta_t)_\xi
=F_1.
\end{aligned}
\end{equation}
where
\begin{equation}\label{3.23}
\begin{aligned}
F_1&=\frac{1}{2}[(\frac{\hat{p}}{\hat{v}})_t+s_-(\frac{\hat{p}}{\hat{v}})_\xi]\phi_\xi^2+
\frac{C_v}{2}[(\frac{1}{\hat{\theta}})_t+s_-(\frac{1}{\hat{\theta}})_\xi]\zeta_\xi^2\\
&-(\frac{\hat{p}}{\hat{v}})_\xi\phi_t\phi_\xi-C_v(\frac{1}{\hat{\theta}})_\xi\zeta_t\zeta_\xi
+(p-\hat{p})\psi_\xi\frac{\zeta_{\xi\xi}}{\hat{\theta}}+\hat{u}_\xi(p-\hat{p})\frac{\zeta_{\xi\xi}}{\hat{\theta}}\\
&-\kappa(\frac{1}{v})_\xi\zeta_\xi\frac{\zeta_{\xi\xi}}{\hat{\theta}}
+\kappa(\frac{\hat{\theta}_\xi\phi}{v\hat{v}})_\xi\frac{\zeta_{\xi\xi}}{\hat{\theta}}
+G_1\psi_{\xi\xi}+G_2\frac{\zeta_{\xi\xi}}{\hat{\theta}}-(\frac{(R\zeta-\hat{p}\phi)\phi}{v\hat{v}})_\xi\psi_{\xi\xi}\\
&+(\frac{\hat{p}}{\hat{v}})_\xi\phi_\xi\psi_\xi
-(\frac{R}{\hat{v}})_\xi\psi_\xi\zeta_\xi+[(\frac{1}{\hat{v}})_\xi(R\zeta-\hat{p}\phi)
-(\frac{\hat{p}_\xi}{\hat{v}})\phi]\psi_{\xi\xi}.
\end{aligned}
\end{equation}
It is easy to see that

\begin{equation}\label{3.24}
\begin{aligned}
|F_1|&\lesssim(|\bar{v}_\xi|+|\tilde{v}_\xi|+N(T))(\phi_\xi^2+\psi_\xi^2+\zeta_\xi^2+\psi_{\xi\xi}^2+\zeta_{\xi\xi}^2)\\[2mm]
&+(|\tilde{u}_\xi|+|\bar{u}_\xi|)(\phi^2+\zeta^2)+\hat{v}_\xi G_2^2+G_2\zeta_{\xi\xi}+G_1\psi_{\xi\xi},
\end{aligned}
\end{equation}

and
\begin{equation}\label{3.25}
\begin{aligned}
\IT\IXX \hat{v}_\xi G_2^2d\xi d\tau
\lesssim & \IT\IXX(\tilde{v}_\xi+\bar{v}_\xi)(\bar{\delta}e^{-c(\xi+t)}+|\tilde{\theta}_{\xi\xi}|^2
+|\tilde{\theta}_\xi|^4)d\xi d\tau \\
\lesssim &\bar{\delta}+\IT(\|\tilde{\theta}_{\xi\xi}\|_{L^\infty}^2+\|\tilde{\theta}_\xi\|_{L^\infty}^4)
(\|\bar{v}_\xi\|_{L^1}+\|\tilde{v}_\xi\|_{L^1})d\tau  \\
\lesssim &\bar{\delta}+\epsilon^{\frac{1}{8}}\IT(1+\tau)^{-\frac{39}{28}}d\tau \\
\lesssim & \bar{\delta}+\epsilon^{\frac{1}{8}},
\end{aligned}
\end{equation}

\begin{equation}\label{3.26}
\begin{aligned}
&\IT\IXX G_2\zeta_{\xi\xi}+G_1\psi_{\xi\xi}d\xi d\tau \\
\lesssim & \bar{\delta}\IT\IXX e^{-c(\xi+t)}(\zeta_{\xi\xi}+\psi_{\xi\xi})d\xi d\tau+\IT\IXX(|\tilde{\theta}_{\xi\xi}|+|\tilde{\theta}_\xi|^2)\zeta_{\xi\xi}d\xi d\tau \\
\lesssim & \bar{\delta}(1+\IT\|(\zeta_{\xi\xi},\psi_{\xi\xi})(\tau)\|^2d\tau)
+\IT(\|\tilde{\theta}_{\xi\xi}\|+\|\tilde{\theta}_\xi\|_{L^4}^2)\|\zeta_{\xi\xi}(\tau)\|d\tau \\
\lesssim & \bar{\delta}(1+\IT\|(\zeta_{\xi\xi},\psi_{\xi\xi})(\tau)\|^2d\tau)
+\IT \epsilon^{\frac{1}{8}}(1+\tau)^{-\frac{3}{4}}\|\zeta_{\xi\xi}(\tau)\|d\tau \\
\lesssim & \bar{\delta}(1+\IT\|(\zeta_{\xi\xi},\psi_{\xi\xi})(\tau)\|^2d\tau)
+ \epsilon^{\frac{1}{8}}\IT[(1+\tau)^{-\frac{3}{2}}+\|\zeta_{\xi\xi}(\tau)\|^2]d\tau \\
\lesssim &(\bar{\delta}+\epsilon^{\frac{1}{8}})\{1+\IT\|(\zeta_{\xi\xi},\psi_{\xi\xi})(\tau)\|^2d\tau\}.
\end{aligned}
\end{equation}
Integrating (\ref{3.22}) over $[0,t]\times\mathbb{R}_+$, and making use of (\ref{3.9}), (\ref{3.15}), (\ref{3.24})-(\ref{3.26}), we get
\begin{equation}\label{3.27}
\begin{aligned}
&\|(\phi_\xi,\psi_\xi,\zeta_\xi)(t)\|^2+\IT \|\frac{\zeta_{\xi\xi}}{\sqrt{v\hat{\theta}}}(\tau)\|^2d\tau \\
&+\IT[\frac{|s_-|}{2}(\frac{p_-}{v_-}\phi_\xi^2+\psi_\xi^2+C_v\frac{\zeta_\xi^2}{\theta_-})
-\frac{p_-}{v_-}\phi_\xi\psi_\xi+\frac{R}{v_-}\zeta_\xi\psi_\xi](\tau,0)d\tau \\
\lesssim & \|(\phi_0,\psi_0,\zeta_0)\|_1^2
+\bar{\delta}+\eps^{\frac{1}{8}}+(\epsilon^{\frac{1}{8}}
+\bar{\delta}+N(T))\IT\|(\phi_\xi,\psi_\xi)(\tau)\|_1^2d\tau.
\end{aligned}
\end{equation}

Then we should deal with the boundary terms.
Since $z_-\in\Omega_{+}$, see (\ref{2.24}), that is, $R\theta_-<u^2_-<\gamma R\theta_-$,
the discriminant of the quadratic form
\begin{equation*}
\frac{|s_-|}{2}(\frac{p_-}{v_-}\phi_\xi^2+\psi_\xi^2)-\frac{p_-}{v_-}\phi_\xi\psi_\xi
\end{equation*}
is less than zero, i.e.
\begin{equation}\label{3.28}
\begin{aligned}
D&=(\frac{p_-}{v_-})^2-4\times\frac{|s_-|}{2}\frac{p_-}{v_-}\times\frac{|s_-|}{2}\\
&=(\frac{p_-}{v_-})(\frac{p_-}{v_-}-|s_-|^2)=(\frac{p_-}{v_-})(\frac{R\theta_-}{v_-^2}-\frac{u_-^2}{v_-^2})<0,
\end{aligned}
\end{equation}
thus, the binomial expression is positive, we get for some constant $c_0>0$ such that
\begin{equation}\label{3.29}
\begin{aligned}
\IT[\frac{|s_-|}{2}(\frac{p_-}{v_-}\phi_\xi^2+\psi_\xi^2+C_v\frac{\zeta_\xi^2}{\theta_-})-\frac{p_-}{v_-}\phi_\xi\psi_\xi]
(\tau,0) d\tau
\geq c_0 \IT (\phi_\xi^2+\psi_\xi^2+\zeta_\xi^2)(\tau,0)d\tau.
\end{aligned}
\end{equation}
Secondly by the Sobolev inequality, it holds that
\begin{equation}\label{3.30}
\begin{aligned}
&\IT\frac{R}{v_-}(\zeta_\xi\psi_\xi)(\tau,0)d\tau \\[2mm]
\lesssim & \frac{c_0}{4}\IT\psi_\xi^2(\tau,0)d\tau+\IT\|\zeta_\xi(\tau)\|^2_{\infty}d\tau\\[2mm]
\lesssim & \frac{c_0}{4}\IT\psi_\xi^2(\tau,0)d\tau
+\frac{1}{4}\IT\|\frac{\zeta_{\xi\xi}}{\sqrt{v\hat{\theta}}}(\tau)\|^2d\tau+\IT\|\zeta_\xi(\tau)\|^2d\tau.
\end{aligned}
\end{equation}
Inserting (\ref{3.29}), (\ref{3.30}) into (\ref{3.27}) and using the result of (\ref{3.9}),
 we get the estimate of (\ref{3.21}) and  complete the proof of Lemma 3.2.
\end{proof}

As for $\int^t_0\|(\phi_\xi,\psi_\xi)(\tau)\|^2d\tau$, we have following Lemma.

\begin{Lemma}
Under the same assumptions listed in Proposition 3.2,
if $\bar{\delta},\eps, N(T)$ are suitably small, it holds that
\begin{eqnarray}\label{3.31}
\begin{aligned}
&\int^t_0\|(\phi_\xi,\psi_\xi)(\tau)\|^2 d\tau \\[2mm]
\lesssim &\|(\phi_0,\psi_0,\zeta_0)\|^2_1+\bar{\delta}+\epsilon^{\frac{1}{8}}+
(\bar{\delta}+\epsilon^{\frac{1}{8}}+N(T))\IT\|(\phi_{\xi\xi},\psi_{\xi\xi})(\tau)\|^2d\tau.
\end{aligned}
\end{eqnarray}
\end{Lemma}
\begin{proof}
$(\ref{3.2})_2$ multiplies $-\frac{\widehat{p}}{2}\phi_\xi$, it holds that

\begin{eqnarray}\label{3.32}
\begin{aligned}
&-(\frac{\widehat{p}}{2}\phi_\xi \psi )_t +(\frac{\widehat{p}}{2})_t\psi\phi_\xi
+ (\frac{\widehat{p}}{2}\phi_t \psi )_\xi- \frac{\widehat{p}_\xi}{2} \psi(s_-\phi_\xi+\psi_\xi)
-\frac{\widehat{p}}{2}\psi^2_\xi+\frac{\widehat{p}^2}{2v}\phi^2_\xi\\[2mm]
&-(\frac{R}{v}\zeta)_\xi\frac{\widehat{p}}{2}\phi_\xi+ \frac{\widehat{p}}{2}\phi_\xi(\frac{\widehat{p}}{v})_\xi\phi=G_1\frac{\widehat{p}}{2}\phi_\xi.
\end{aligned}
\end{eqnarray}

$(\ref{3.2})_3$ multiplies $\psi_\xi$, it holds that

\begin{eqnarray}\label{3.33}
\begin{aligned}
&(C_v\zeta \psi_\xi )_t - (C_v\zeta \psi_t )_\xi- C_v\zeta_\xi((p-\widehat{p})_\xi+G_1)+\widehat{p}\psi^2_\xi\\[2mm]
&=\kappa (\frac{\zeta_\xi}{v}- \frac{\widehat{\theta}_\xi \phi}{\widehat{v} v})_\xi\psi_\xi
-(p-\widehat{p}) \widehat{u}_\xi\psi_\xi-(p-\widehat{p})\psi^2_\xi-G_2\psi_\xi.
\end{aligned}
\end{eqnarray}
Combining together, we get

\begin{eqnarray}\label{3.34}
\begin{aligned}
&(C_v\zeta \psi_\xi -\frac{\widehat{p}}{2}\phi_\xi \psi )_t -(C_v\zeta \psi_t-\frac{\widehat{p}}{2}\phi_t \psi )_\xi
+\frac{\widehat{p}}{2}\psi^2_\xi+\frac{\widehat{p}^2}{2v}\phi^2_\xi\\[2mm]
= & \frac{\widehat{p}_\xi}{2} \psi(s_-\phi_\xi+\psi_\xi)-(\frac{\widehat{p}}{2})_t\psi\phi_\xi+(\frac{R}{v}\zeta)_\xi\frac{\widehat{p}}{2}\phi_\xi- \frac{\widehat{p}}{2}\phi_\xi(\frac{\widehat{p}}{v})_\xi\phi
+C_v\zeta_\xi((p-\widehat{p})_\xi+G_1)\\[2mm]
&+\kappa (\frac{\zeta_\xi}{v}- \frac{\widehat{\theta}_\xi \phi}{\widehat{v} v})_\xi\psi_\xi
-(p-\widehat{p}) \widehat{u}_\xi\psi_\xi-(p-\widehat{p})\psi^2_\xi+G_1\frac{\widehat{p}}{2}\phi_\xi-G_2\psi_\xi\\[2mm]
\lesssim & O(1)|\widehat{u}_\xi||(\phi,\psi,\xi)||(\phi_\xi,\psi_\xi,\zeta_\xi)|+O(1)(\overline{\delta}+ \eps^\frac{1}{8}+N(t))|(\phi^2_\xi,\psi^2_\xi,\zeta^2_\xi)|\\[2mm]
&+|\phi_\xi \zeta_\xi|+|\psi_\xi \zeta_{\xi\xi}|+|G_1\phi_\xi|+|G_2\psi_\xi|
+|G_1\zeta_\xi|+|\zeta_\xi|^2.
\end{aligned}
\end{eqnarray}
Integrating above euqation over $[0,t]\times\mathbb{R}_+$, it holds
that

\begin{eqnarray}\label{3.35}
\begin{aligned}
&\int^t_0\int_{\mathbb{R}_+}(\phi^2_\xi+\psi^2_\xi) d\xi d\tau\\[2mm]
\lesssim & \|(\phi_\xi,\psi_\xi,\psi,\zeta)(t)\|^2+ \|(\phi_{0\xi},\psi_{0\xi},\psi_0,\zeta_0)(t)\|^2+\bar{\delta}+\eps^{\frac{1}{8}}\\[2mm]
&+ \int^t_0\|(\zeta_\xi,\zeta_{\xi\xi})(\tau)\|^2d\tau
 +  \int^t_0\int_{\mathbb{R}_+}|\widehat{u}_\xi||(\phi,\xi)|^2d\xi dt\\[2mm]
&+(\frac{1}{4}+\overline{\delta}+\eps^{\frac{1}{8}} +N(T))\int^t_0\int_{\mathbb{R}_+}(\phi^2_\xi+\psi^2_\xi) d\xi d\tau.
\end{aligned}
\end{eqnarray}

Using the result of (\ref{3.9}), (\ref{3.15}) and (\ref{3.21}) into (\ref{3.35}), we could get (\ref{3.31}) under our assumptions and complete the proof of Lemma 3.3.
\end{proof}

Combining the results of Lemma 3.1-Lemma 3.3, we get
\begin{equation}\label{3.39}
\begin{aligned}
&\|(\phi,\psi,\zeta)(t)\|^2_1+\int^t_0(\psi_\xi^2+\phi_\xi^2+\zeta_\xi^2)(\tau,0)d\tau
+\int^t_0\|(\phi_\xi,\psi_\xi,\zeta_\xi,\zeta_{\xi\xi})(\tau)\|^2d\tau \\[2mm]
\lesssim & \|(\phi_0,\psi_0,\zeta_0)\|_1^2
+\bar{\delta}+\epsilon^{\frac{1}{8}}+(\bar{\delta}+\epsilon^{\frac{1}{8}}+N(T))
\IT\|(\phi_{\xi\xi},\psi_{\xi\xi})(\tau)\|^2d\tau.
\end{aligned}
\end{equation}

To control the higher boundary terms later,  we need the estimates of the normal direction.

\begin{Lemma}Under the same  assumptions listed in Proposition 3.2, if $\bar{\delta}, \epsilon, N(T)$ are suitably small, it holds that
\begin{equation}\label{3.50}
\begin{aligned}
&\|(\phi_t,\psi_t,\zeta_t)(t)\|^2+\IT\|\zeta_{t\xi}(\tau)\|^2d\tau\\
\lesssim &\|(\phi_0,\psi_0,\zeta_0)\|_2^2+\|(\phi_t,\psi_t,\zeta_t)(0)\|^2+\bar{\delta}+\epsilon^{\frac{1}{8}}
+(\bar{\delta}+\epsilon^{\frac{1}{8}}+N(T))\IT\|(\phi_{\xi\xi},\psi_{\xi\xi})(\tau)\|^2.
\end{aligned}
\end{equation}
\end{Lemma}
\begin{proof} Just let $(\ref{3.2})_{1t}\times\frac{\hat{p}}{\hat{v}}\phi_{t}, (\ref{3.2})_{2t}\times\psi_{t}, (\ref{3.2})_{3t}\times\frac{\zeta_{t}}{\hat{\theta}}$, we get that
\begin{equation}\label{3.51}
\begin{aligned}
&\frac{1}{2}(\frac{\hat{p}}{\hat{v}}\phi_t^2+\psi_t^2+C_v\frac{\zeta_t^2}{\hat{\theta}})_t
-\frac{s_-}{2}(\frac{\hat{p}}{\hat{v}}\phi_t^2+\psi_t^2+C_v\frac{\zeta_t^2}{\hat{\theta}})_\xi\\[2mm]
&-(\frac{\hat{p}}{\hat{v}}\phi_t\psi_t-\frac{R}{\hat{v}}\zeta_t\psi_t)_\xi
-\kappa((\frac{\zeta_\xi}{v}-\frac{\hat{\theta}_\xi\phi}{v\hat{v}})_t\frac{\zeta_t}{\hat{\theta}})_\xi+\kappa\frac{\zeta_{t\xi}^2}{v\hat{\theta}}
=F_3
\end{aligned}
\end{equation}
where
\begin{equation}\label{3.52}
\begin{aligned}
|F_3|&\lesssim(|\tilde{v}_\xi|+|\bar{v}_\xi|+N(T))
(\phi_t^2+\phi_\xi^2+\psi_t^2+\psi_\xi^2+\zeta_\xi^2+\zeta_t^2+\phi_{t\xi}^2+\psi_{t\xi}^2
+\zeta_{t\xi}^2)\\[2mm]
&+|G_{1t}\psi_t|+|G_{2t}\zeta_t|+|\tilde{v}_\xi|(\phi^2+\zeta^2)\\
&\lesssim(|\tilde{v}_\xi|+|\bar{v}_\xi|+N(T))
(\phi_\xi^2+\psi_\xi^2+\zeta_\xi^2+\phi_{\xi\xi}^2+\psi_{\xi\xi}^2+\zeta_{\xi\xi}^2+G_1^2+G_2^2+G_{1\xi}^2
+\zeta_{t\xi}^2)\\[2mm]
&+G_{1t}^2+G_{2t}^2+(|\tilde{v}_\xi|+|\bar{v}_\xi|)(\phi^2+\zeta^2).
\end{aligned}
\end{equation}

Integrating (\ref{3.51}) over $[0,t]\times \mathbb{R}_+$, and noticing that $(\phi_t,\psi_t,\zeta_t)(t,0)=(0,0,0)$.
By using the estimates (\ref{3.39}) and previous results, we get (\ref{3.50}), we omit the details.
\end{proof}

\begin{Lemma} Under the same assumptions listed in Proposition 3.2, if $\bar{\delta},\epsilon, N(T)$ are suitably small,
 it holds that
\begin{equation}\label{3.54}
\begin{aligned}
&\|(\phi_{t\xi},\psi_{t\xi},\zeta_{t\xi})(t)\|^2
+\IT(\phi_{t\xi}^2+\psi_{t\xi}^2+\zeta_{t\xi}^2)(0,\tau)d\tau
+\IT\|\zeta_{t\xi\xi}(\tau)\|^2d\tau\\
\lesssim &\|(\phi_0,\psi_0,\zeta_0)\|_2^2+\|(\phi_t,\psi_t,\xi_t)(0)\|^2_1
+\bar{\delta}+\epsilon^{\frac{1}{8}}\\[2mm]
&+(\bar{\delta}+\epsilon^{\frac{1}{8}}+N(T))\IT\|(\phi_{\xi\xi},\psi_{\xi\xi}
,\zeta_{\xi\xi\xi})(\tau)\|^2.
\end{aligned}
\end{equation}
\end{Lemma}

\begin{proof} Let $(\ref{3.2})_{1t}\times-\frac{\hat{p}}{\hat{v}}\phi_{t\xi\xi}, (\ref{3.2})_{2t}\times-\psi_{t\xi\xi}, (\ref{3.2})_{3t}\times-\frac{\zeta_{t\xi\xi}}{\hat{\theta}}$, we get that
\begin{equation}\label{3.55}
\begin{aligned}
&\frac{1}{2}(\frac{\hat{p}}{\hat{v}}\phi_{t\xi}^2+\psi_{t\xi}^2+C_v\frac{\zeta_{t\xi}^2}{\hat{\theta}})_t
+\kappa\frac{\zeta_{t\xi\xi}^2}{v\hat{\theta}}\\[2mm]
&+\{\frac{s_-\hat{p}}{2\hat{v}}\phi_{t\xi}^2+\frac{s_-}{2}\psi_{t\xi}^2
+\frac{s_-}{2}C_v\frac{\zeta_{t\xi}^2}{\hat{\theta}}+\frac{\hat{p}}{\hat{v}}\phi_{t\xi}\psi_{t\xi}
-\frac{R}{v}\zeta_{t\xi}\psi_{t\xi}\}_\xi\\
&-(\frac{\hat{p}}{\hat{v}}\phi_{tt}\phi_{t\xi}+\psi_{tt}\psi_{t\xi}+C_v\zeta_{tt}
\frac{\zeta_{t\xi}}{\hat{\theta}})_\xi=F_4,
\end{aligned}
\end{equation}
where
\begin{equation}\label{3.56}
\begin{aligned}
&F_4=-(\frac{\hat{p}}{\hat{v}})_\xi\phi_{tt}\phi_{t\xi}+\frac{1}{2}((\frac{\hat{p}}{\hat{v}})_t+s_-
(\frac{\hat{p}}{\hat{v}})_\xi)\phi_{t\xi}^2+\frac{1}{2}C_v((\frac{1}{\hat{\theta}})_t
+s_-(\frac{1}{\hat{\theta}})_\xi)\zeta_{t\xi}^2\\
&-C_v(\frac{1}{\hat{\theta}})_\xi\zeta_{tt}\zeta_{t\xi}+(\frac{\hat{p}}{\hat{v}})_\xi\phi_{t\xi}\psi_{t\xi}
-(\frac{R}{v})_\xi\zeta_{t\xi}\psi_{t\xi}+\frac{\hat{p}\phi}{v\hat{v}}\phi_{t\xi}\psi_{t\xi\xi}
+\frac{p\zeta}{\theta\hat{\theta}}\psi_{t\xi}\zeta_{t\xi\xi}\\
&+((\frac{R}{v})_t\zeta_\xi+(\frac{R}{v})_\xi\zeta_t+(\frac{R}{v})_{t\xi}\zeta
-(\frac{\hat{p}}{v})_t\phi_\xi-(\frac{\hat{p}}{v})_\xi\phi_t
-(\frac{\hat{p}}{v})_{t\xi}\phi+G_{1t})\psi_{t\xi\xi}\\
&+(p_t\psi_\xi+(\hat{u}_\xi(p-\hat{p}))_t+k(\frac{\hat{\theta}_\xi\phi}{v\hat{v}})_{t\xi}-
\kappa[(\frac{1}{v})_{t\xi}\zeta_\xi+(\frac{1}{v})_t\zeta_{\xi\xi}+(\frac{1}{v})_\xi\zeta_{t\xi}]
+G_{2t})\frac{\zeta_{t\xi\xi}}{\hat{\theta}}
\end{aligned}
\end{equation}
 Using the relationship
$\psi_{t\xi\xi}=\phi_{tt\xi}-s_-\phi_{t\xi\xi}$ and  previous estimates,  we have

\begin{equation}\label{3.57}
\begin{aligned}
&|\IT\int_{R_+}F_4dxd\tau|\\
\lesssim &\|(\phi_0,\psi_0,\zeta_0)\|_2^2+\|(\phi_t,\psi_t,\zeta_t(0)\|_1^2+\bar{\delta}+\eps^{\frac{1}{8}}
+N(T)\|\phi_{t\xi}\|^2+\bar{\delta}\IT\psi_{t\xi}^2(\tau,0)d\tau
\\[2mm]
&+(\bar{\delta}+\eps^{\frac{1}{8}}+N(T))\IT\|(\phi_{\xi\xi},\psi_{\xi\xi},\zeta_{\xi\xi\xi},\zeta_{t\xi\xi})(\tau)\|^2d\tau.
\end{aligned}
\end{equation}
Integrating (\ref{3.55}) over $[0,t]\times \mathbb{R}_+$ and note that $\partial^i_t(\phi,\psi,\zeta)(\tau,0)=(0,0,0)(i=1,2)$.
By using (\ref{3.39}), we get
\begin{equation}\label{3.58}
\begin{aligned}
&\|(\phi_{t\xi},\psi_{t\xi},\zeta_{t\xi})(t)\|^2+\IT\|\frac{\zeta_{t\xi\xi}}{v\hat{\theta}}(\tau)\|^2d\tau\\
&+\IT\{\frac{|s_-|}{2}(\frac{p_-}{v_-}\phi_{t\xi}^2+\psi_{t\xi}^2
+C_v\frac{\zeta_{t\xi}^2}{\theta_-})
-\frac{p_-}{v_-}\phi_{t\xi}\psi_{t\xi}-\frac{R}{v_-}\zeta_{t\xi}\psi_{t\xi})\}(\tau,0)d\tau\\
\lesssim &\|(\phi_0,\psi_0,\zeta_0)\|_2^2+\|(\phi_t,\psi_t,{\zeta}_t)(0)\|^2_1
+\bar{\delta}\IT\psi_{t\xi}^2(\tau,0)d\tau+\bar{\delta}+\eps^{\frac{1}{8}}\\
&+(\bar{\delta}+\eps^{\frac{1}{8}}+N(T))
\IT\|(\phi_{\xi\xi},\psi_{\xi\xi},\zeta_{\xi\xi\xi})(\tau)\|^2d\tau.
\end{aligned}
\end{equation}

Similar as (\ref{3.29}) and (\ref{3.30}), we have
\begin{equation}\label{3.59}
\begin{aligned}
&\IT[\frac{|s_-|}{2}(\frac{p_-}{v_-}\phi_{t\xi}^2+\psi_{t\xi}^2
+C_v\frac{\zeta_{t\xi}^2}{\theta_-})-\frac{p_-}{v_-}\phi_{t\xi}\psi_{t\xi}](\tau,0)d\tau\\[2mm]
\geq &c_0 \IT(\phi_{t\xi}^2+\psi_{t\xi}^2+\zeta_{t\xi}^2)(\tau,0)d\tau,
\end{aligned}
\end{equation}
and
\begin{equation}\label{3.60}
\begin{aligned}
&\IT\frac{R}{v_-}(\zeta_{t\xi}\psi_{t\xi})(\tau,0)d\tau \\
\lesssim & \frac{c_0}{4} \IT
\psi_{t\xi}^2(\tau,0)d\tau+\IT\|\zeta_{t\xi}(\tau)\|^2d\tau
+\frac{1}{4}\IT\|\frac{\zeta_{t\xi\xi}}{\sqrt{v\hat{\theta}}}(\tau)\|^2d\tau,
\end{aligned}
\end{equation}
inserting (\ref{3.59}) and (\ref{3.60}) into (\ref{3.58}) and using (\ref{3.50}),
 we finally get
(\ref{3.54}) and complete the proof of Lemma 3.5.
\end{proof}

By these results prepared, we can deal with the higher order estimates.

\begin{Lemma}Under the same assumptions listed in Proposition 3.2, if $\bar{\delta},  \epsilon, N(T)$ are suitably small,  then it holds that
\begin{equation}\label{3.62}
\begin{aligned}
&\|(\phi_{\xi\xi},\psi_{\xi\xi},\zeta_{\xi\xi})(t)\|^2
+\IT|(\phi_{\xi\xi},\psi_{\xi\xi},\zeta_{\xi\xi})|^2(\tau,0)d\tau
+\IT\|\frac{\zeta_{\xi\xi\xi}}{v\hat{\theta}}(\tau)\|^2d\tau\\
\lesssim &\|(\phi_0,\psi_0,\zeta_0)\|_2^2+\|(\phi_t,\psi_t,\xi_t)(0)\|^2_1+\bar{\delta}+\epsilon^{\frac{1}{8}}\\[2mm]
&+(\bar{\delta}+\epsilon^{\frac{1}{8}}+N(T))\IT\|(\phi_\xi,\psi_\xi)(\tau)\|_1^2d\tau.
\end{aligned}
\end{equation}
\end{Lemma}
\begin{proof} Multiplying $(\ref{3.2})_{1\xi}$ by $ -\frac{\hat{p}}{\hat{v}}\phi_{\xi\xi\xi}$, $(\ref{3.2})_{2\xi}$ by $-\psi_{\xi\xi\xi}$, $(\ref{3.2})_{3\xi}$ by $-\frac{\zeta_{\xi\xi\xi}}{\hat{\theta}}$ and adding the results, we can get

\begin{equation}\label{3.63}
\begin{aligned}
&\frac{1}{2}(\frac{\hat{p}}{\hat{v}}\phi_{\xi\xi}^2+\psi_{\xi\xi}^2+C_v\frac{\zeta_{\xi\xi}^2}{\hat{\theta}})_t
+\frac{s_-}{2}(\frac{\hat{p}}{\hat{v}}\phi_{\xi\xi}^2+\psi_{\xi\xi}^2
+C_v\frac{\zeta_{\xi\xi}^2}{\hat{\theta}})_\xi  \\
&+(\frac{\hat{p}}{\hat{v}}\phi_{\xi\xi}\psi_{\xi\xi}-\frac{R}{\hat{v}}\psi_{\xi\xi}\zeta_{\xi\xi})_\xi
-(\frac{\hat{p}}{\hat{v}}\phi_{t\xi}\phi_{\xi\xi}+\psi_{t\xi}\psi_{\xi\xi}
+C_v\zeta_{t\xi}\frac{\zeta_{\xi\xi}}{\hat{\theta}})_\xi
+\kappa\frac{\zeta_{\xi\xi\xi}^2}{v\hat{\theta}} \\
=& F_2+G_{2\xi}\frac{\zeta_{\xi\xi\xi}}{\hat{\theta}}+G_{1\xi}\psi_{\xi\xi\xi}
+(\frac{R}{\hat{v}})_{\xi\xi}\zeta\psi_{\xi\xi\xi}-(\frac{\hat{p}}{\hat{v}})_{\xi\xi}\phi\psi_{\xi\xi\xi} \\
&+2(\frac{R}{\hat{v}})_\xi\zeta_\xi\psi_{\xi\xi\xi}-2(\frac{\hat{p}}{\hat{v}})_\xi\phi_\xi\psi_{\xi\xi\xi}
-(\frac{(R\zeta-\hat{p}\phi)\phi}{v\hat{v}})_{\xi\xi}\psi_{\xi\xi\xi},
\end{aligned}
\end{equation}
where
\begin{equation}\label{3.63-1}
\begin{aligned}
|F_2|&\lesssim (|\hat{v}_\xi|+N(T))|(\phi_\xi,\psi_\xi,\zeta_\xi,\phi_{\xi\xi},\psi_{\xi\xi},\zeta_{\xi\xi},\zeta_{\xi\xi\xi})|^2 \\[2mm]
&+(|\bar{u}_\xi|+|\tilde{u}_\xi|)(\phi^2+\zeta^2)+|\hat{v}_\xi|G_{2\xi}^2.
\end{aligned}
\end{equation}
Firstly, we should deal with the high derivative terms. Integration by parts, we have
\begin{equation}\label{3.64}
\begin{aligned}
&\IT\IXX \{G_{1\xi}\psi_{\xi\xi\xi}
+(\frac{R}{\hat{v}})_{\xi\xi}\zeta\psi_{\xi\xi\xi}-(\frac{\hat{p}}{\hat{v}})_{\xi\xi}\phi\psi_{\xi\xi\xi}
+2(\frac{R}{\hat{v}})_\xi\zeta_\xi\psi_{\xi\xi\xi}-2(\frac{\hat{p}}{\hat{v}})_\xi\phi_\xi\psi_{\xi\xi\xi}\}
d\xi d\tau \\
=&\IT\IXX [G_{1\xi}\psi_{\xi\xi}
+(\frac{R}{\hat{v}})_{\xi\xi}\zeta\psi_{\xi\xi}-(\frac{\hat{p}}{\hat{v}})_{\xi\xi}\phi\psi_{\xi\xi}
+2(\frac{R}{\hat{v}})_\xi\zeta_\xi\psi_{\xi\xi}-2(\frac{\hat{p}}{\hat{v}})_\xi\phi_\xi\psi_{\xi\xi}]_\xi d\xi d\tau \\
&-\IT\IXX[G_{1\xi}+(\frac{R}{\hat{v}})_{\xi\xi}\zeta-(\frac{\hat{p}}{\hat{v}})_{\xi\xi}\phi
+2(\frac{R}{\hat{v}})_\xi\zeta_\xi-2(\frac{\hat{p}}{\hat{v}})_\xi\phi_\xi]_\xi\psi_{\xi\xi} \\
\lesssim & (\bar{\delta}+\epsilon)\IT(e^{-c\tau}+\zeta_\xi(\tau,0)+\phi_\xi(\tau,0))\psi_{\xi\xi}(\tau,0)d\tau
+\IT\IXX|G_{1\xi\xi}\psi_{\xi\xi}|d\xi d\tau \\
&+\IT\IXX|\hat{v}_\xi|(\phi_\xi^2+\zeta_\xi^2+\phi_{\xi\xi}^2
+\psi_{\xi\xi}^2+\zeta_{\xi\xi}^2)+(|\hat{v}_\xi|^2+|\hat{v}_{\xi\xi}|)(\phi^2+\zeta^2)d\xi d\tau \\
\lesssim &(\bar{\delta}+\epsilon)[1+\IT(\zeta_\xi^2+\phi_\xi^2+\psi_{\xi\xi}^2)(\tau,0)d\tau]
+\IT\IXX|G_{1\xi\xi}\psi_{\xi\xi}|d\xi d\tau \\
&+(\bar{\delta}+\eps)\IT\|(\phi_\xi,\psi_\xi,\zeta_\xi)(\tau)\|_1^2d\tau
+\IT\IXX(|\hat{v}_\xi|^2+|\hat{v}_{\xi\xi}|)(\phi^2+\zeta^2)d\xi d\tau.
\end{aligned}
\end{equation}
Integrating (\ref{3.63}) and making use of (\ref{3.63-1}), (\ref{3.64}), we get

\begin{equation}\label{3.65}
\begin{aligned}
&\|(\phi_{\xi\xi},\psi_{\xi\xi},\zeta_{\xi\xi})(t)\|^2+\IT\|\frac{\zeta_{\xi\xi\xi}}{v\hat{\theta}}(\tau)\|^2d\tau
+\IT J(\tau,0)d\tau \\
\lesssim & \|(\phi_0,\psi_0,\zeta_0)\|_2^2+(\bar{\delta}+\epsilon^\frac{1}{8})\{1+\IT(\zeta_\xi^2+\phi_\xi^2+\psi_{\xi\xi}^2)(\tau,0)d\tau\}\\
&+(\bar{\delta}+\eps^{\frac{1}{8}}+N(T))\IT\|(\phi_\xi,\psi_\xi)(\tau)\|_1^2d\tau\\
&+\IT\IXX|\hat{v}_\xi|^2G_{2\xi}^2+|G_{2\xi}\frac{\zeta_{\xi\xi\xi}}{\hat{\theta}}|
+|G_{1\xi\xi}\psi_{\xi\xi}|d\xi d\tau   \\
&+\IT\IXX(|\bar{u}_\xi|+|\tilde{u}_\xi|+|\tilde{v}_\xi\bar{v}_\xi|+|\tilde{v}_\xi|^2+|\tilde{v}_{\xi\xi}|)
(\phi^2+\zeta^2)d\xi d\tau \\
&-\IT\IXX(\frac{(R\zeta-\hat{p}\phi)\phi}{v\hat{v}})_{\xi\xi}\psi_{\xi\xi\xi}d\xi d\tau,
\end{aligned}
\end{equation}
where
\begin{eqnarray*}
\begin{aligned}
J:&=-\frac{s_-p_-}{2 v_-} \phi^2_{\xi\xi}
-\frac{s_-}{2} \psi^2_{\xi\xi}-\frac{p_-}{v_-}\phi_{\xi\xi}\psi_{\xi\xi }+\frac{R}{v_-} \zeta_{\xi\xi}\psi_{\xi\xi}
-\frac{s_- C_v}{2 \theta} \zeta^2_{\xi\xi}\\[2mm]
&+\frac{p_-}{v_-} \phi_{t\xi} \phi_{\xi\xi}+ \psi_{t\xi} \psi_{\xi\xi} +\frac{C_v}{\theta_-} \zeta_{\xi\xi} \zeta_{t\xi}.
\end{aligned}
\end{eqnarray*}

For the last term on the right hand side of (\ref{3.65}), we estimate it in the following
\begin{equation}\label{3.66}
\begin{aligned}
&\IT\IXX(\frac{(R\zeta-\hat{p}\phi)\phi}{v\hat{v}})_{\xi\xi}\psi_{\xi\xi\xi}d\xi d\tau \\
=&\IT\IXX[\frac{(R\zeta-2\hat{p}\phi)}{v\hat{v}}
-\frac{(R\zeta-\hat{p}\phi)\phi}{\hat{v}v^2}]\phi_{\xi\xi}
\psi_{\xi\xi\xi}+2[(\frac{R\zeta-\hat{p}\phi}{v\hat{v}})_\xi\phi_\xi\psi_{\xi\xi}]_\xi \\
&-2[(\frac{R\zeta-\hat{p}\phi}{v\hat{v}})_\xi\phi_\xi]_\xi\psi_{\xi\xi}
+[\frac{(R\zeta-\hat{p}\phi)\phi}{\hat{v}}(2\frac{v_\xi^2}{v^3}-\frac{\hat{v}_{\xi\xi}}{v^2})\psi_{\xi\xi}]_\xi\\
&-[\frac{(R\zeta-\hat{p}\phi)\phi}{\hat{v}}(2\frac{v_\xi^2}{v^3}-\frac{\hat{v}_{\xi\xi}}{v^2})]_\xi\psi_{\xi\xi}
+2[(\frac{1}{v})_\xi(\frac{R\zeta-\hat{p}\phi}{\hat{v}})_\xi\phi\psi_{\xi\xi}]_\xi \\
&-2[(\frac{1}{v})_\xi(\frac{R\zeta-\hat{p}\phi}{\hat{v}})_\xi\phi]_\xi\psi_{\xi\xi}
+(\frac{1}{v}[(\frac{R\zeta}{\hat{v}})_{\xi\xi}+2(\frac{\hat{p}}{\hat{v}})_\xi\phi_\xi
+(\frac{\hat{p}}{\hat{v}})_{\xi\xi}\phi]\phi\psi_{\xi\xi})_\xi \\
&-\{(\frac{1}{v}[(\frac{R\zeta}{\hat{v}})_{\xi\xi}+2(\frac{\hat{p}}{\hat{v}})_\xi\phi_\xi
+(\frac{\hat{p}}{\hat{v}})_{\xi\xi}\phi]\phi\}_\xi\psi_{\xi\xi}d\xi d\tau \\
=&\IT\IXX\{[\frac{(R\zeta-2\hat{p}\phi)}{v\hat{v}}
-\frac{(R\zeta-\hat{p}\phi)\phi}{\hat{v}v^2}]\phi_{\xi\xi}
\psi_{\xi\xi\xi}+H\}d\xi d\tau,
\end{aligned}
\end{equation}
where
\begin{equation}\label{3.67}
\begin{aligned}
\IT\IXX Hd\xi d\tau  &\lesssim (\overline{\delta}+\epsilon^{\frac{1}{8}})\{1
+\IT(\phi_\xi^2+\psi_{\xi\xi}^2)(\tau,0)d\tau \}\\
&+(\epsilon^{\frac{1}{8}}+\bar{\delta}+N(T))\IT(\|(\phi_\xi,\psi_\xi,\zeta_\xi)(\tau)\|_1^2+\|\zeta_{\xi\xi\xi}(\tau)\|^2)d\tau,
\end{aligned}
\end{equation}
and
\begin{equation}\label{3.68}
\begin{aligned}
&\IT\IXX[\frac{(R\zeta-2\hat{p}\phi)}{v\hat{v}}
-\frac{(R\zeta-\hat{p}\phi)\phi}{\hat{v}v^2}]\phi_{\xi\xi}\psi_{\xi\xi\xi}d\xi d\tau \\
=&\IT\IXX[\frac{(R\zeta-2\hat{p}\phi)}{v\hat{v}}
-\frac{(R\zeta-\hat{p}\phi)\phi}{\hat{v}v^2}]\phi_{\xi\xi}(\phi_{t\xi\xi}-s_-\phi_{\xi\xi\xi})d\xi d\tau \\
=&\IT\IXX[\frac{1}{2}(\frac{(R\zeta-2\hat{p}\phi)}{v\hat{v}}
-\frac{(R\zeta-\hat{p}\phi)\phi}{\hat{v}v^2})\phi_{\xi\xi}^2]_t \\
&-[\frac{1}{2}s_-(\frac{(R\zeta-2\hat{p}\phi)}{v\hat{v}}
-\frac{(R\zeta-\hat{p}\phi)\phi}{\hat{v}v^2})\phi_{\xi\xi}^2]_\xi \\
&-[\frac{1}{2}(\frac{(R\zeta-2\hat{p}\phi)}{v\hat{v}}
-\frac{(R\zeta-\hat{p}\phi)\phi}{\hat{v}v^2})]_t\phi_{\xi\xi}^2 \\
&+[\frac{1}{2}s_-(\frac{(R\zeta-2\hat{p}\phi)}{v\hat{v}}
-\frac{(R\zeta-\hat{p}\phi)\phi}{\hat{v}v^2})]_\xi\phi_{\xi\xi}^2 d\xi d\tau \\
\lesssim &\|(\phi_0,\psi_0,\zeta_0)\|_2^2+\bar{\delta}+\epsilon^\frac{1}{8}
+N(T)\|\phi_{\xi\xi}\|^2\\
&+(\bar{\delta}+\epsilon^\frac{1}{8}+N(T))\{\IT\|(\phi_{\xi\xi},\zeta_{\xi\xi\xi})
(\tau)\|^2d\tau\},
\end{aligned}
\end{equation}
Here we use the fact $\phi_{t\xi\xi}-s_-\phi_{\xi\xi\xi}-\psi_{\xi\xi\xi}=0$ in (\ref{3.68}).
Similar as (\ref{3.25}),(\ref{3.26}), we have
\begin{equation}\label{3.69}
\begin{aligned}
&\IT\IXX\hat{v}_\xi G_{2\xi}^2+|G_{2\xi}\frac{\zeta_{\xi\xi\xi}}{\hat{\theta}}|
+|G_{1\xi\xi}\psi_{\xi\xi}|d\xi d\tau  \\
\lesssim &(\overline{\delta}+\epsilon^{\frac{1}{8}})\{1+\IT\|(\psi_{\xi\xi},\zeta_{\xi\xi\xi})(\tau)\|^2d\tau\}.
\end{aligned}
\end{equation}
Making use of (\ref{3.9}), (\ref{3.15}), (\ref{3.17})  again, we have
\begin{equation}\label{3.70}
\begin{aligned}
&\IT\IXX(|\bar{u}_\xi|+|\tilde{u}_\xi|+|\tilde{v}_\xi\bar{v}_\xi|+|\tilde{v}_\xi|^2+|\tilde{v}_{\xi\xi}|)
(\phi^2+\zeta^2)d\xi d\tau \\
\lesssim &\|(\phi_0,\psi_0,\zeta_0)\|_2^2+(\epsilon^{\frac{1}{8}}+\overline{\delta})
\{1+N(T)+\IT\|(\phi_\xi,\psi_\xi,\zeta_\xi)(\tau)\|^2d\tau\}.
\end{aligned}
\end{equation}
Inserting (\ref{3.66})-(\ref{3.70}) into (\ref{3.65}) and
using (\ref{3.39}), it holds that
\begin{equation}\label{3.71}
\begin{aligned}
&\|(\phi_{\xi\xi},\psi_{\xi\xi},\zeta_{\xi\xi})(t)\|^2
+\IT\|\frac{\zeta_{\xi\xi\xi}}{v\hat{\theta}}(\tau)\|^2d\tau
+\IT J(\tau,0)d\tau \\
\lesssim &\|(\phi_0,\psi_0,\zeta_0)\|_2^2+(\bar{\delta}+\eps^{\frac{1}{8}})\{1
+\int_{0}^{t}\psi_{\xi\xi}^2(\tau,0)d\tau\}\\
&+(\bar{\delta}+\epsilon^{\frac{1}{8}}+N(T))
\IT\|(\phi_\xi,\psi_\xi)(\tau)\|_1^2d\tau.
\end{aligned}
\end{equation}
Then, we should estmiate the boundary term $J$.
Differentiating $(\ref{3.2})_1$ and  $(\ref{3.2})_2$ by t and  choosing $\xi=0$, by the boundary condition
\begin{eqnarray}\label{3.73}
\begin{aligned}
\partial^i_t(\phi,\psi,\zeta)(t,0)=(0,0,0),(i=0,1,2),
\end{aligned}
\end{eqnarray}
we get
\begin{eqnarray}\label{3.74}
\left\{
\begin{array}{lll}
-s_-\phi_{t\xi}(t,0)-\psi_{t\xi}(t,0)=0,  \\[2mm]
-s_-\psi_{t\xi}(t,0)-\frac{p_-}{v_-}\phi_{t\xi}(t,0)=-\frac{R}{v_-}\zeta_{t\xi}(t,0),
\end{array}
\right.
\end{eqnarray}
which means that
\begin{eqnarray}\label{3.75}
\begin{aligned}
&\psi_{t\xi}(t,0)=-s_-\phi_{t\xi}(t,0), \\[2mm]
&\frac{R}{v_-}\zeta_{t\xi}(t,0)=(-s_-^2+\frac{p_-}{v_-})\phi_{t\xi}(t,0).
\end{aligned}
\end{eqnarray}
On the other hand, differentiating $(\ref{3.2})_1$ and  $(\ref{3.2})_2$ by $\xi$ and  choosing $\xi=0$, it holds
\begin{eqnarray}\label{3.76}
\left\{
\begin{array}{l}
\phi_{t\xi}(t,0)=s_-\phi_{\xi\xi}(t,0)+\psi_{\xi\xi}(t,0),  \\[2mm]
\psi_{t\xi}(t,0)=(s_-\psi_{\xi\xi}+\frac{p_-}{v_-}\phi_{\xi\xi}-\frac{R}{v_-}\zeta_{\xi\xi})(t,0)+O(1)(\bar{\delta}+N(t)+\eps^{\frac{1}{8}})|(\phi_\xi,\zeta_\xi)(t,0)|\\[2mm]
+O(1)\bar{\delta}e^{-ct}.
\end{array}
\right.
\end{eqnarray}
Reminding that $s_-<0$,  we divided the boundary term $J(t,0)$ into three parts:
\begin{eqnarray*}
\begin{aligned}
J_1&:=-\frac{s_-p_-}{2 v_-} \phi^2_{\xi\xi}(t,0)-\frac{p_-}{v_-} \phi_{\xi\xi}\psi_{\xi\xi}(t,0)
-\frac{s_-}{2} \psi^2_{\xi\xi}(t,0)-\frac{s_- C_v}{2\theta_-}\zeta^2_{\xi\xi}(t,0),\\[2mm]
J_2&:=\frac{R}{v_-} \zeta_{\xi\xi}\psi_{\xi\xi}(t,0)
+\frac{C_v}{\theta_-} \zeta_{\xi\xi} \zeta_{t\xi}(t,0),\\[2mm]
J_3&:=\frac{p_-}{v_-} \phi_{t\xi} \phi_{\xi\xi}(t,0)+ \psi_{t\xi} \psi_{\xi\xi}(t,0).
\end{aligned}
\end{eqnarray*}
Similar as (\ref{3.29}), we can get

\begin{eqnarray}\label{3.76-1}
\begin{aligned}
\int^t_0 J_1(\tau,0)d\tau\geq c_0\int^t_0 (\phi^2_{\xi\xi}+\psi^2_{\xi\xi}+\zeta^2_{\xi\xi})(\tau,0)d\tau.
\end{aligned}
\end{eqnarray}
As for $J_2$, by (\ref{3.75}) and (\ref{3.76}), it holds that
\begin{eqnarray}\label{3.77}
\begin{aligned}
&\int^t_0 J_2(\tau,0)d\tau
\lesssim \int^t_0|\zeta_{\xi\xi}||(\phi_{\xi\xi},\psi_{\xi\xi})|(\tau,0)d\tau \\[2mm]
&\lesssim \frac{c_0}{4}\int^t_0(\phi^2_{\xi\xi}+\psi^2_{\xi\xi})(\tau,0)d\tau
+\frac{1}{4}\int^t_0\|\frac{\zeta_{\xi\xi\xi}}{v\hat{\theta}}(\tau)\|^2d\tau+\int^t_0\|\zeta_{\xi\xi}(\tau)\|^2d\tau,
\end{aligned}
\end{eqnarray}
As for $J_3$,
since $\psi_{t\xi}(t,0)=-s_-\phi_{t\xi}(t,0)$, it holds that
\begin{eqnarray}\label{3.78}
\begin{aligned}
J_3&=\phi_{t\xi}(\frac{p_-}{v_-} \phi_{\xi\xi}-s_-\psi_{\xi\xi})(t,0)\\[2mm]
&=\phi_{t\xi}(-s_-\psi_{\xi\xi}-s^2_-\phi_{\xi\xi})(t,0)+(\frac{p_-}{v_-}+s^2_-) \phi_{t\xi}\phi_{\xi\xi}(t,0)\\[2mm]
&=-s_-\phi^2_{t\xi}(t,0)+(\frac{p_-}{v_-}+s^2_-)\frac{R}{p_--s^2_-v_-} \phi_{\xi\xi}\zeta_{t\xi}(t,0).
\end{aligned}
\end{eqnarray}
Inserting (\ref{3.76-1})-(\ref{3.78}) into (\ref{3.71}) and using (\ref{3.39}), then we get
\begin{equation}\label{3.79}
\begin{aligned}
&\|(\phi_{\xi\xi},\psi_{\xi\xi},\zeta_{\xi\xi})(t)\|^2+\IT\|\frac{\zeta_{\xi\xi\xi}}{v\hat{\theta}}(\tau)\|^2d\tau
+ \frac{c_0}{4}\int^t_0(\phi^2_{\xi\xi}+\psi^2_{\xi\xi}+\zeta^2_{\xi\xi})(\tau,0)d\tau\\[2mm]
\lesssim &\|(\phi_0,\psi_0,\zeta_0)\|_2^2+\bar{\delta}+\eps^{\frac{1}{8}}
+(\bar{\delta}+\epsilon^{\frac{1}{8}}+N(T))
\IT\|(\phi_\xi,\psi_\xi)(\tau)\|_1^2d\tau
+\int^t_0(\phi_{\xi\xi}\zeta_{t\xi})(\tau,0)d\tau\\[2mm]
\lesssim &\|(\phi_0,\psi_0,\zeta_0)\|_2^2+\bar{\delta}+\eps^{\frac{1}{8}}
+(\bar{\delta}+\epsilon^{\frac{1}{8}}+N(T))
\IT\|(\phi_\xi,\psi_\xi)(\tau)\|_1^2d\tau\\[2mm]
&+\frac{c_0}{8}\int^t_0\phi^2_{\xi\xi}(\tau,0)d\tau+\int^t_0\|(\zeta_{t\xi},\zeta_{t\xi\xi})(\tau)\|^2d\tau,
\end{aligned}
\end{equation}
Using the results of Lemma 3.4-Lemma 3.5, under our smallness assumptions,  we can get (\ref{3.62}) and complete the proof of Lemma 3.6.
\end{proof}

Combining the results of Lemma 3.1-Lemma 3.6, we get
\begin{equation}\label{3.80}
\begin{aligned}
&\|(\phi,\psi,\zeta)(t)\|^2_2+\|(\phi_t,\psi_t,\zeta_t)(t)\|^2_1
+\int^t_0\|(\phi_\xi,\psi_\xi)(\tau)\|_1^2+\|\zeta_\xi(\tau)\|_2^2d\tau \\[2mm]
&+\int^t_0\|(\psi_{t\xi},\phi_{t\xi},\zeta_{t\xi},\zeta_{t\xi\xi})(\tau)\|^2d\tau
+\int^t_0(\psi_\xi^2+\phi_\xi^2+\zeta_\xi^2)(\tau,0)d\tau\\[2mm]
&+\int^t_0(\psi_{\xi\xi}^2+\phi_{\xi\xi}^2+\zeta_{\xi\xi}^2+\psi_{t\xi}^2+\phi_{t\xi}^2+\zeta_{t\xi}^2)(\tau,0)d\tau \\[2mm]
\lesssim & \|(\phi_0,\psi_0,\zeta_0)\|_2^2+\|(\phi_t,\psi_t,\xi_t)(0)\|^2_1+\bar{\delta}+\epsilon^{\frac{1}{8}}\\[2mm]
&+(\bar{\delta}+\epsilon^{\frac{1}{8}}+N(T))\IT\|(\phi_{\xi\xi},\psi_{\xi\xi})(\tau)\|^2d\tau.
\end{aligned}
\end{equation}

At last, we turn to  estimate $\IT\|(\phi_{\xi\xi},\psi_{\xi\xi})(\tau)\|^2d\tau$.

\begin{Lemma}Under the same
assumptions listed in Proposition 3.2, if $\delta, \eps, N(T)$ are suitably small,  it holds that
\begin{eqnarray}\label{3.83}
\begin{aligned}
\int^t_0\|(\phi_{\xi\xi},\psi_{\xi\xi})(\tau)\|^2d\tau
\lesssim \|(\phi_0,\psi_0,\zeta_0)\|_2^2+\|(\phi_t,\psi_t,\xi_t)(0)\|^2_1+\bar{\delta}+\epsilon^{\frac{1}{8}}.
\end{aligned}
\end{eqnarray}
\end{Lemma}
\begin{proof}
$(\ref{3.2})_{2\xi}$ multiplies $-\frac{\widehat{p}}{2}\phi_{\xi\xi}$, it holds that
\begin{eqnarray}\label{3.84}
\begin{aligned}
&-(\frac{\widehat{p}}{2}\phi_{\xi\xi} \psi_\xi )_t + (\frac{\widehat{p}}{2}\phi_{t\xi} \psi_\xi )_\xi- \frac{\widehat{p}_\xi}{2} \psi_\xi(s_-\phi_{\xi\xi}+\psi_{\xi\xi})
-\frac{\widehat{p}}{2}\psi^2_{\xi\xi}+\frac{\widehat{p}^2}{2 v}\phi^2_{\xi\xi}\\[2mm]
=&-\frac{\widehat{p}_t}{2}\phi_{\xi\xi}\psi_\xi+\big\{(\frac{R}{v}\zeta)_{\xi\xi}- (\frac{\widehat{p}}{v})_{\xi\xi}\phi-2(\frac{\widehat{p}}{v})_\xi\phi_\xi+G_{1\xi}\big\}\frac{\widehat{p}}{2}\phi_{\xi\xi}.
\end{aligned}
\end{eqnarray}
$(\ref{3.2})_{3\xi}$ multiplies $\psi_{\xi\xi}$, it holds that
\begin{eqnarray}\label{3.85}
\begin{aligned}
&(C_v\zeta_\xi \psi_{\xi\xi} )_t - (C_v\zeta_\xi\psi_{t\xi} )_\xi+\widehat{p}\psi^2_{\xi\xi}+\widehat{p}_\xi\psi_\xi\psi_{\xi\xi}\\[2mm]
=&\big\{\kappa (\frac{\zeta_\xi}{v}- \frac{\widehat{\theta}_\xi \phi}{\widehat{v} v})_\xi
-(p-\widehat{p}) u_\xi-G_2\big\}_\xi\psi_{\xi\xi}+C_v\zeta_{\xi\xi}\big[(p-\widehat{p})_\xi+G_1\big]_\xi.
\end{aligned}
\end{eqnarray}
Combining (\ref{3.84}) and (\ref{3.85}) together, we get
\begin{eqnarray}\label{3.86}
\begin{aligned}
&(C_v\zeta_\xi \psi_{\xi\xi}-\frac{\widehat{p}}{2}\phi_{\xi\xi} \psi_\xi )_t - (C_v\zeta_\xi\psi_{t\xi} -\frac{\widehat{p}}{2}\phi_{t\xi} \psi_\xi)_\xi+\frac{\widehat{p}}{2}\psi^2_{\xi\xi}+\frac{\widehat{p}^2}{2 v}\phi^2_{\xi\xi}\\[2mm]
=&-\frac{\widehat{p}_t}{2}\phi_{\xi\xi}\psi_\xi+\frac{\widehat{p}_\xi}{2} \psi_\xi(s_-\phi_{\xi\xi}-\psi_{\xi\xi})
+\big\{(\frac{R}{v}\zeta)_{\xi\xi}- (\frac{\widehat{p}}{v})_{\xi\xi}\phi-2(\frac{\widehat{p}}{v})_\xi\phi_\xi+G_{1\xi}\big\}\frac{\widehat{p}}{2}\phi_{\xi\xi}\\[2mm]
&+\big\{\kappa (\frac{\zeta_\xi}{v}- \frac{\widehat{\theta}_\xi \phi}{\widehat{v} v})_\xi
-(p-\widehat{p}) u_\xi-G_2\big\}_\xi\psi_{\xi\xi}+C_v\zeta_{\xi\xi}\big[(p-\widehat{p})_\xi+G_1\big]_\xi.
\end{aligned}
\end{eqnarray}
Integrating it over $[0,t]\times\mathbb{R}_+$, by  (\ref{3.76}), it holds that
\begin{eqnarray}\label{3.87}
\begin{aligned}
&\int^t_0\int_{\mathbb{R}_+}\|(\phi_{\xi\xi},\psi_{\xi\xi})(\tau)\|^2 d\tau\\[2mm]
\lesssim & \|(\phi_0,\psi_0,\zeta_0)\|_2^2+\|(\phi_t,\psi_t,\xi_t)(0)\|^2_1+\bar{\delta}+\eps^{\frac{1}{8}}
+  \int^t_0\int_{\mathbb{R}_+}|\widehat{u}_\xi||(\phi,\xi)|^2d\xi dt \\[2mm]
&+\int_0^t\|\zeta_{\xi}(\tau)\|_2^2d\tau+\int_0^t(\psi_\xi^2+\phi_{\xi\xi}^2+\psi_{\xi\xi}^2)(\tau,0)d\tau\\[2mm]
&+(\frac{1}{4}+\bar{\delta} +\eps^{\frac{1}{8}}+N(T))\int^t_0\|(\phi_\xi,\psi_\xi)(\tau)\|^2_1 d\tau.
\end{aligned}
\end{eqnarray}
Using the result of (\ref{3.9}), (\ref{3.15}) and (\ref{3.80}), we get (\ref{3.83}) and complete the proof of Lemma 3.7.
\end{proof}

Combining the results of Lemma 3.1-Lemma 3.7,   we finally get (\ref{3.6}) and complete the proof of Proposition 3.2.\\

 Acknowledgments: The authors are grateful to Professor S.Nishibata, Professors Feimin Huang and Huijiang Zhao with
their support and advices. This work was supported by the Fundamental Research
 grants from the  Science
Foundation of Hubei Province under the contract 2018CFB693.

\end{document}